\documentclass[a4paper]{amsart}
\usepackage[english]{babel}
\usepackage[active]{srcltx}
\usepackage{amsmath,amssymb,amsthm}

\usepackage[pdftex]{color}
\usepackage[bookmarks=true,hyperindex,pdftex,colorlinks,citecolor=blue]
{hyperref}
\hypersetup{pdftitle={The Bishop-Phelps-Bollob\'{a}s moduli of a Banach space}, pdfauthor={Chica, Kadets, Martin, Moreno, Rambla}}

\usepackage{graphicx}

\usepackage{pgf,tikz}
\usetikzlibrary{arrows}

\theoremstyle{plain}
\newtheorem{theorem}{Theorem}[section]
\newtheorem{corollary}[theorem]{Corollary}
\newtheorem{prop}[theorem]{Proposition}
\newtheorem{lemma}[theorem]{Lemma}

\theoremstyle{definition}

\newtheorem{remark}[theorem]{Remark}
\newtheorem{example}[theorem]{Example}
\newtheorem{examples}[theorem]{Examples}
\newtheorem{definition}[theorem]{Definition}
\newtheorem{definitions}[theorem]{Definitions}

 \DeclareMathOperator{\re}{Re\,}

 \DeclareMathOperator{\dist}{dist\,}
  \DeclareMathOperator{\lin}{Lin\,}

\newcommand{\C}{\mathbb{C}}
\newcommand{\R}{\mathbb{R}}
\newcommand{\N}{\mathbb{N}}

\newcommand{\norm}[1]{\ensuremath{\lVert#1\rVert}}

\newcommand{\abs}[1]{\ensuremath{\lvert#1\rvert}}

\newcommand{\eps}{\varepsilon}

\renewcommand{\leq}{\leqslant}
\renewcommand{\geq}{\geqslant}

\begin{document}
\title[Bishop-Phelps-Bollob\'{a}s moduli of a Banach space]{Bishop-Phelps-Bollob\'{a}s moduli of a Banach space}

\dedicatory{Dedicated to the memory of Robert R.\ Phelps (1926--2013)}

\author[Chica]{Mario Chica}
\author[Kadets]{Vladimir Kadets}
\author[Mart\'{\i}n]{Miguel Mart\'{\i}n}
\author[Moreno-Pulido]{Soledad Moreno-Pulido}
\author[Rambla-Barreno]{Fernando Rambla-Barreno}

\address[Chica \& Mart\'{\i}n]{Departamento de An\'{a}lisis Matem\'{a}tico\\
Facultad de Ciencias\\
Universidad de Granada\\
18071 Granada, Spain}

\email{\texttt{mcrivas@ugr.es} \qquad \texttt{mmartins@ugr.es}}

\address[Kadets]{Department of Mechanics and Mathematics\\ Kharkiv V.N. Karazin National University\\ 61022 Kharkiv, Ukraine}
\email{\texttt{vova1kadets@yahoo.com}}

\address[Moreno-Pulido \& Rambla-Barreno]{Departamento de Matem\'{a}ticas\\ Universidad de C\'{a}diz\\ Puerto Real (C\'{a}diz), Spain}

\email{\texttt{soledad.moreno@uca.es} \qquad \texttt{fernando.rambla@uca.es}}

\thanks{First author partially supported by Spanish MINECO and FEDER project no.\ MTM2012-31755 and by Junta de Andaluc\'{\i}a and FEDER grant FQM-185. Second author partially supported by Junta de Andaluc\'{\i}a and FEDER grants FQM-185 and P09-FQM-4911 and by the program GENIL-PRIE of the CEI of the University of Granada. Third author partially supported by Spanish MINECO and FEDER project no.\ MTM2012-31755, and by Junta de Andaluc\'{\i}a and FEDER grants FQM-185 and P09-FQM-4911. Fourth and Fifth authors partially supported by Junta de Andaluc\'{\i}a and FEDER grant FQM-257.}

\begin{abstract}
We introduce two Bishop-Phelps-Bollob\'{a}s moduli of a Banach space which measure, for a given Banach space, what is the best possible Bishop-Phelps-Bollob\'{a}s theorem in this space. We show that there is a common upper bound for these moduli for all Banach spaces and we present an example showing that this bound is sharp. We prove the continuity of these moduli and an inequality with respect to duality. We calculate the two moduli for Hilbert spaces and also present many examples for which the moduli have the maximum possible value (among them, there are $C(K)$ spaces and $L_1(\mu)$ spaces). Finally, we show that if a Banach space has the maximum possible value of any of the moduli, then it contains almost isometric copies of the real space $\ell_\infty^{(2)}$ and present an example showing that this condition is not sufficient.
\end{abstract}

\date{March 27th, 2013}

\maketitle \thispagestyle{empty}

\section{Introduction}

The classical Bishop-Phelps theorem of 1961 \cite{Bishop-Phelps} states that the set of norm attaining functionals on a Banach space is norm dense in the dual space. Few years later, B.~Bollob\'{a}s
\cite{Bollobas} gave a sharper version of this theorem allowing to
approximate at the same time a functional and a vector in which it
almost attains the norm (see the result bellow). The main aim of
this paper is to study the best possible approximation of this
kind that one may have in each Banach space, measuring it by using
two moduli which we define.

Before going further, we first present the original result by
Bollob\'{a}s which nowadays is known as the
\emph{Bishop-Phelps-Bollob\'{a}s theorem}. We need to fix
some notation. Given a (real or complex) Banach space $X$, we
write $B_X$ and $S_X$ to denote the closed unit ball and the unit
sphere of the space, and $X^*$ denotes the (topological) dual of
$X$. We will also use the notation
$$
\Pi(X):=\bigl\{(x,x^*)\in X\times X^*\,:\,
\|x\|=\|x^*\|=x^*(x)=1\bigr\}.
$$

\begin{theorem}[Bishop-Phelps-Bollob\'{a}s theorem \cite{Bollobas}]$ $\newline
Let $X$ be a Banach space. Suppose $x\in S_X$ and $x^*\in S_{X^*}$
satisfy $ |1-x^*(x)|\leq \eps^2/2 $ ($0<\eps<1/2$). Then there
exists $(y,y^*)\in \Pi(X)$ such that $\|x-y\|<\eps + \eps^2$ and
$\|x^*-y^*\|\leq \eps$.
\end{theorem}

So the idea is that given $(x,x^*)\in S_X\times S_{X^*}$ such that
$x^*(x)\sim 1$, there exist $y\in S_X$ close to $x$ and $y^*\in
S_{X^*}$ close to $x^*$ for which $y^*(y)=1$. This result has many
applications, especially for the theory of numerical ranges, see
\cite{Bollobas,B-D2}.

Our objective is to introduce two moduli which measures, for a
given Banach space, what is the best possible
Bollob\'{a}s theorem in this space, that is, how close can be $y$
to $x$ and $y^*$ to $x^*$ in the result above depending on how
close is $x^*(x)$ to $1$. In the first modulus, we allow the vector and the functional to have norm less than or equal to one, whereas in the second modulus we only consider norm-one vectors and functionals.

\begin{definitions}[Bishop-Phelps-Bollob\'{a}s modulus] \label{def1.2bpb-mod}$ $\newline Let $X$ be a Banach space. The \emph{Bishop-Phelps-Bollob\'{a}s modulus} of $X$ is the function $\Phi_X:(0,2)\longrightarrow \R^+$ such that given $\delta\in (0,2)$, $\Phi_X(\delta)$ is the infimum of those $\eps>0$ satisfying that for every $(x,x^*)\in B_X\times B_{X^*}$ with $\re x^*(x)>1-\delta$, there is $(y,y^*)\in \Pi(X)$ with
$\|x-y\|<\eps$ and $\|x^*-y^*\|<\eps$.

The \emph{spherical Bishop-Phelps-Bollob\'{a}s modulus} of $X$ is the function $\Phi_X^S:(0,2)\longrightarrow \R^+$ such that given $\delta\in (0,2)$, $\Phi^S_X(\delta)$ is the infimum of those $\eps>0$ satisfying that for every $(x,x^*)\in S_X\times S_{X^*}$ with $\re x^*(x)>1-\delta$, there is $(y,y^*)\in \Pi(X)$ with
$\|x-y\|<\eps$ and $\|x^*-y^*\|<\eps$.
\end{definitions}

Evidently, $\Phi_X^S(\delta) \leq
\Phi_X(\delta)$, so any estimation from above for
$\Phi_X(\delta)$ is also valid for $\Phi_X^S(\delta)$ and, viceversa, any estimation from below for $\Phi_X^S(\delta)$ is also valid for $\Phi_X(\delta)$.

Recall that the dual of a complex Banach space $X$ is isometric
(taking real parts) to the dual of the real subjacent space
$X_\R$. Also, $\Pi(X)$ does not change if we consider $X$ as a real Banach space (indeed, if $(x,x^*)\in \Pi(X)$ then $x^*\in S_{X^*}$
and $x\in S_X$ satisfies $x^*(x)=1$ so, obviously, $\re x^*(x)=1$
and $(x,\re x^*)\in \Pi(X_\R)$). Therefore, only the real
structure of the space is playing a role in the above definitions.
We could then suppose that we are only dealing with real Banach
spaces and any result would apply automatically to complex spaces.
Nevertheless, we are not going to do so, mainly because for
classical sequence or function spaces, the real space underlying
the complex version of the space is not equal, in general, to the real version of the space. We then prefer to develop the theory for real and complex spaces which, actually, does not suppose much more effort. Unless otherwise is stated, the (arbitrary or concrete) spaces we are dealing with will be real or complex and the results work in both cases.

Some notation will help to the understanding and further use of
Definitions~\ref{def1.2bpb-mod}. Let $X$ be a Banach space and fix $0<\delta<2$.
Writing
\begin{align*}
A_X(\delta)&:=\bigl\{(x,x^*)\in B_X\times B_{X^*}\,:\, \re
x^*(x)>1-\delta \bigr\},\\
A^S_X(\delta)&:=\bigl\{(x,x^*)\in S_X\times S_{X^*}\,:\, \re
x^*(x)>1-\delta \bigr\},
\end{align*}
it is clear that
\begin{align*}
\Phi_X(\delta)& =\sup\limits_{(x,x^*)\in A_X(\delta)}\
\inf\limits_{(y,y^*)\in \Pi(X)}\ \max\{\|x-y\|,\|x^*-y^*\|\},\\
\Phi_X^S(\delta)& =\sup\limits_{(x,x^*)\in A_X^S(\delta)}\
\inf\limits_{(y,y^*)\in \Pi(X)}\ \max\{\|x-y\|,\|x^*-y^*\|\}.
\end{align*}
Therefore, if we write $d_H(A,B)$ to denote the Hausdorff distance
between $A,B \subset X\times X^*$ associated to the
$\ell_\infty$-distance $d_\infty$ in $X\times X^*$ (that is, $
d_\infty\bigl((x,x^*),(y,y^*)\bigr)=\max\{\|x-y\|,\|x^*-y^*\|\} $
for $(x,x^*),(y,y^*)\in X\times X^*$, and
$$
d_H(A,B)=\max\left\{\sup\limits_{a\in A}\,\inf\limits_{b\in B}
d_\infty(a,b)\, , \, \sup\limits_{b\in B}\,\inf\limits_{a\in A}
d_\infty(a,b) \right\}
$$
for $A,B\subset X\times X^*$), then we clearly have that
\begin{align*}
\Phi_X(\delta)=d_H\bigl(A_X(\delta),\Pi(X)\bigr) \qquad \text{and} \qquad \Phi_X^S(\delta)=d_H\bigl(A^S_X(\delta),\Pi(X)\bigr)
\end{align*}
for every $0<\delta<2$
(observe that $\Pi(X)\subset A_X(\delta)$ and $\Pi(X)\subset A^S_X(\delta)$ for every $\delta$).

The following result is immediate.

\begin{remark}
{\slshape Let $X$ be a Banach space. Given $\delta_1,\delta_2\in
(0,2)$ with $\delta_1 <\delta_2$, one has
$$
A_X(\delta_1)\subset A_X(\delta_2) \quad \text{and} \quad A_X^S(\delta_1)\subset A_X^S(\delta_2).
$$
Therefore, the functions $\Phi_X(\cdot)$ and $\Phi^S_X(\cdot)$ are increasing.}
\end{remark}

Routine computations and the fact that the Hausdorff distance
does not change if we take closure in one of the sets, provide the
following observations.

\begin{remark}\label{remark_1.4}
{\slshape Let $X$ be a Banach space. Then, for every
$\delta\in(0,2)$, one has}
\begin{align*}
\Phi_X(\delta) & := \inf\bigl\{\eps>0 \,: \, \forall (x,x^*)\in B_{X}\times B_{X^*} \text{ with }\re x^*(x)>1-\delta,\\ & \qquad \qquad \qquad \quad \exists (y,y^*)\in \Pi(X) \text{ with } d_\infty((x,x^*),(y,y^*))<\eps   \bigr\} \\
& = \inf\bigl\{\eps>0 \,: \, \forall (x,x^*)\in B_{X}\times B_{X^*} \text{ with }\re x^*(x)\geq 1-\delta,\\ & \qquad \qquad \qquad \quad \exists (y,y^*)\in \Pi(X) \text{ with } d_\infty((x,x^*),(y,y^*))<\eps   \bigr\}\\
& = \inf\bigl\{\eps>0 \,: \, \forall (x,x^*)\in B_{X}\times B_{X^*} \text{ with }\re x^*(x)>1-\delta,\\ & \qquad \qquad \qquad \quad \exists (y,y^*)\in \Pi(X) \text{ with } d_\infty((x,x^*),(y,y^*))\leq \eps   \bigr\}\\
& = \inf\bigl\{\eps>0 \,: \, \forall (x,x^*)\in B_{X}\times
B_{X^*} \text{ with }\re x^*(x)\geq 1-\delta,\\ & \qquad \qquad
\qquad \quad \exists (y,y^*)\in \Pi(X) \text{ with }
d_\infty((x,x^*),(y,y^*))\leq \eps   \bigr\},\\
\intertext{ {\slshape and} }
\Phi_X^S(\delta) & := \inf\bigl\{\eps>0 \,: \, \forall (x,x^*)\in S_{X}\times S_{X^*} \text{ with }\re x^*(x)>1-\delta,\\ & \qquad \qquad \qquad \quad \exists (y,y^*)\in \Pi(X) \text{ with } d_\infty((x,x^*),(y,y^*))<\eps   \bigr\} \\
& = \inf\bigl\{\eps>0 \,: \, \forall (x,x^*)\in S_{X}\times S_{X^*} \text{ with }\re x^*(x)\geq 1-\delta,\\ & \qquad \qquad \qquad \quad \exists (y,y^*)\in \Pi(X) \text{ with } d_\infty((x,x^*),(y,y^*))<\eps   \bigr\}\\
& = \inf\bigl\{\eps>0 \,: \, \forall (x,x^*)\in S_{X}\times S_{X^*} \text{ with }\re x^*(x)>1-\delta,\\ & \qquad \qquad \qquad \quad \exists (y,y^*)\in \Pi(X) \text{ with } d_\infty((x,x^*),(y,y^*))\leq \eps   \bigr\}\\
& = \inf\bigl\{\eps>0 \,: \, \forall (x,x^*)\in S_{X}\times
S_{X^*} \text{ with }\re x^*(x)\geq 1-\delta,\\ & \qquad \qquad
\qquad \quad \exists (y,y^*)\in \Pi(X) \text{ with }
d_\infty((x,x^*),(y,y^*))\leq \eps   \bigr\}.
\end{align*}
\end{remark}

Observe that {\slshape the smaller are the functions
$\Phi_X(\cdot)$ and $\Phi_X^S(\cdot)$, the better is the approximation on the space}. It
can be deduced from the Bishop-Phelps-Bollob\'{a}s theorem that
there is a common upper bound for $\Phi_X(\cdot)$ and $\Phi_X^S(\cdot)$ for all Banach
spaces $X$. Our first result in the next section will be to
present the best possible upper bound, namely we will show that
\begin{equation} \label{introd-eq1}
\Phi_X^S(\delta)\leq \Phi_X(\delta)\leq \sqrt{2\delta} \qquad \bigl(0<\delta<2,\ X\
\text{Banach space}\bigr).
\end{equation}
This will follow from a result by R.~Phelps \cite{Phelps}. A
version for $\Phi_X^S(\delta)$ for small $\delta$'s can be also
deduced from the Br{\o}ndsted-Rockafellar variational principle
\cite[Theorem~3.17]{Phelps-libro}, as claimed in \cite{CasGuiKad}.
The sharpness of \eqref{introd-eq1} can be verified by considering
the real space $X = \ell_\infty^{(2)}$. This is the content of
section~\ref{sec:2}.

Next, we prove in section~\ref{sec:3} that for every Banach space
$X$, the moduli $\Phi_X(\delta)$ and $\Phi_X^S(\delta)$ are
continuous in $\delta$. We prove that $\Phi_X(\delta)\leq
\Phi_{X^*}(\delta)$ and $\Phi_X^S(\delta)\leq \Phi_{X^*}^S(\delta)$. Finally, we show that $\Phi_X(\delta)=\sqrt{2\delta}$ if and only if $\Phi_X^S(\delta)=\sqrt{2\delta}$.

Examples of spaces for which the two moduli are computed are presented in section~\ref{sec:4}. Among other results, the moduli of $\R$ and of every real or complex Hilbert space of
(real)-dimension greater than one are calculated, and there are presented a number of spaces for which the value of both moduli are
$\sqrt{2\delta}$ (i.e.\ the maximal possible value) for small $\delta$'s: namely $c_0$, $\ell_1$ and, more in general, $L_1(\mu)$, $C_0(L)$, unital $C^*$-algebras with non-trivial centralizer\ldots

The main result of section~\ref{sec:5} states that if a Banach
space $X$ satisfies $\Phi_X(\delta_0)=\sqrt{2\delta_0}$ (equivalently, $\Phi_X^S(\delta_0)=\sqrt{2\delta_0}$) for some
$\delta_0\in (0,1/2)$, then $X$ contains almost isometric
copies of the real space $\ell_\infty^{(2)}$. We finish
presenting, for every $\delta\in (0, 1/2)$, an example of a three dimensional real space $Z$ containing an isometric copy of $\ell_\infty^{(2)}$ for which $\Phi_Z(\delta)<\sqrt{2\delta}$. This is the content of section~\ref{sec:6}.

\newpage

\section{The upper bound of the moduli}\label{sec:2}

Our first result is the promised best upper bound of the
Bishop-Phelps-Bollob\'{a}s moduli.

\begin{theorem}\label{Theorem B-P-B improved}
For every Banach space $X$ and every $\delta\in (0,2)$,
$\Phi_X(\delta)\leq \sqrt{2\delta}$ and so, $\Phi_X^S(\delta)\leq \sqrt{2\delta}$
\end{theorem}

We deduce the above result from \cite[Corollary~2.2]{Phelps},
which was stated for general bounded convex sets on real Banach
spaces. Particularizing the result to the case of the unit ball of
a Banach space, using a routine argument to change non-strict
inequalities to strict inequalities, and taking into account that
the dual of a complex Banach space is isometric (taking real
parts) to the dual of the real subjacent space, we get the
following result.

\begin{prop}[Particular case of \mbox{\cite[Corollary~2.2]{Phelps}}] \label{prop:corolario2.2-phelps}
Let $X$ be Banach space. Suppose that $z^*\in S_{X^*}$, $z\in B_X$
and $\eta>0$ are given such that $\re z^*(z) > 1-\eta$. Then, for
any $k\in(0,1)$ there exist $\tilde{y}^*\in X^*$ and $\tilde{y}
\in S_X$ such that
\begin{equation*}
 \|\tilde{y}^*\|=\tilde{y}^*(\tilde{y}),\qquad \norm{z-\tilde{y}} < \frac{\eta}{k}, \qquad \norm{z^*-\tilde{y}^*} < k.
\end{equation*}
\end{prop}

\begin{proof}[Proof of Theorem~\ref{Theorem B-P-B improved}] We have to show that given $(x,x^*)\in B_X \times B_{X^*}$ with $\re x^*(x)>1-\delta$, there exists $(y,y^*)\in \Pi(X)$ such that $\|x-y\|<\sqrt{2\delta}$ and $\|x^*-y^*\|<\sqrt{2\delta}$.
Let us first prove the result for the more interesting case of
$\delta\in(0,1)$. In this case,
$$
0<1-\delta < \|x^*\|\leq 1,
$$
so, if we write $\eta=\dfrac{\|x^*\|-1+\delta}{\|x^*\|}>0$,
$z^*=x^*/\|x^*\|$ and $z=x$, one has
$$
\re z^*(z)>1-\eta.
$$
Next, we consider $k=\eta/\sqrt{2\delta}$ and claim that $0<k<1$.
Indeed, as the function
\begin{equation}\label{eq:defi-varphi}
\varphi(t)=\dfrac{t-1+\delta}{\sqrt{2\delta}\,t}\qquad (t\in \R^+)
\end{equation}
is strictly increasing, $k=\varphi(\|x^*\|)$ and
$1-\delta<\|x^*\|\leq 1$, we have that
$$
0=\varphi(1-\delta)<k\leq
\varphi(1)=\frac{\sqrt{\delta}}{\sqrt{2}}<1,
$$
as desired. Therefore, we may apply
Proposition~\ref{prop:corolario2.2-phelps} with $z^*\in S_{X^*}$,
$z\in B_X$, $\eta>0$ and $0<k<1$ to obtain $\tilde{y}^*\in X^*$
and $\tilde{y} \in S_X$ satisfying
$$
 \|\tilde{y}^*\|=\tilde{y}^*(\tilde{y}),\qquad \norm{z-\tilde{y}} < \frac{\eta}{k}=\sqrt{2\delta}, \qquad \left\|\frac{x^*}{\|x^*\|}-\tilde{y}^*\right\| < k=\dfrac{\|x^*\|-1+\delta}{\|x^*\|\sqrt{2\delta}}.
$$
As $k<1$, we get $\tilde{y}^*\neq 0$ and we may write $y^*=
\frac{\tilde{y}^*}{\norm{\tilde{y}^*}}$, $y=\tilde{y}$, to get
that $(y,y^*)\in \Pi(X)$. We already have that
$\|x-y\|<\sqrt{2\delta}$. On the other hand, we have
\begin{align*}
\|x^*-y^*\|=\left\|x^*-\frac{\tilde{y}^*}{\|\tilde{y}^*\|}\right\|
& \leq \Bigl\|x^*-\|x^*\|\tilde{y}^*\Bigr\| + \left\|\|x^*\|\tilde{y}^* - \frac{\tilde{y}^*}{\|\tilde{y}^*\|}\right\| \\ &\leq \|x^*\|\left\|\frac{x^*}{\|x^*\|} -\tilde{y}^*\right\| + \bigl|\|x^*\|\|\tilde{y}^*\|-1\bigr| \\ & \leq \|x^*\|\left\|\frac{x^*}{\|x^*\|} -\tilde{y}^*\right\| + \bigl|\|x^*\|\|\tilde{y}^*\|-\|x^*\|\bigr| + \bigl|1-\|x^*\|\bigr| \\
& \leq \|x^*\|\left[\left\|\frac{x^*}{\|x^*\|}
-\tilde{y}^*\right\|+ \bigl|\|\tilde{y}^*\|-1\bigr| \right] + 1 -
\|x^*\| \\ & \leq 2\|x^*\|\left\|\frac{x^*}{\|x^*\|} -
\tilde{y}^*\right\|+1-\|x^*\| \\ & <
\dfrac{2}{\sqrt{2\delta}}\bigl(\|x^*\|-1+\delta\bigr) + 1 -
\|x^*\|.
\end{align*}
Now, as the function
\begin{equation*}
\gamma(t)= \dfrac{2}{\sqrt{2\delta}}\bigl(t-1+\delta\bigr) + 1 - t
\qquad \bigl(t\in [0,1]\bigr)
\end{equation*}
is strictly increasing (for this we only need $0<\delta<2$), we
get $\gamma(\|x^*\|)\leq
\gamma(1)=\frac{2\delta}{\sqrt{2\delta}}=\sqrt{2\delta}$. It
follows that $\|x^*-y^*\|<\sqrt{2\delta}$, as desired.

Let us now prove the case when $\delta\in [1,2)$. Here, it can be
routinely verified that
$$
\frac{\delta-1}{\sqrt{2\delta}-1} < \sqrt{2\delta}-1
$$
so, writing
$$
\psi(\delta)=\frac{1}{2}\left(\frac{\delta-1}{\sqrt{2\delta}-1} +
\sqrt{2\delta}-1\right)
$$
we get
\begin{equation}\label{eq:new-maintheorem-1}
\frac{\delta-1}{\sqrt{2\delta}-1} < \psi(\delta)< \sqrt{2\delta}-1
\qquad \bigl(\delta\in [1,2)\bigr).
\end{equation}
Now, we have to distinguish two situations. Let first suppose that
$\|x^*\|\leq \psi(\delta)$. Then, we take any $y\in S_X$ such that
$\|x-y\|\leq 1$ and take $y^*\in S_{X^*}$ such that $y^*(y)=1$.
Then, $(y,y^*)\in \Pi(X)$, $\|x-y\|\leq 1<\sqrt{2\delta}$ and
$$
\|x^*-y^*\|\leq 1 +\|x^*\| \leq 1 + \psi(\delta) < \sqrt{2\delta}
$$
by \eqref{eq:new-maintheorem-1}. Otherwise, suppose $\|x^*\|>
\psi(\delta)$. We then write
$\eta=\dfrac{\|x^*\|-1+\delta}{\|x^*\|}>0$ and
$k=\eta/\sqrt{2\delta}$ as in the previous case, and we have to
show that $k<1$. This is trivial for the case $\delta=1$ and for
$\delta> 1$, we use that the function $\varphi$ defined in
\eqref{eq:defi-varphi} is now strictly decreasing to get that
$$
k=\varphi(\|x^*\|) <
\varphi\bigl(\psi(\delta)\bigr)<\varphi\left(\dfrac{\delta-1}{\sqrt{2\delta}-1}
\right)=1.
$$
Then, the rest of the proof follows the same lines of the case
when $\delta\in (0,1)$ since this hypothesis is not longer used.
\end{proof}

Let us comment that the above proof is much simpler if we restrict
to $x^*\in S_{X^*}$ (in particular, to the spherical modulus $\Phi^S_X(\delta)$), but the result for non-unital functionals is stronger. Actually, the following stronger version can be deduced by conveniently modifying the election of $k$ in the proof of Theorem~\ref{Theorem B-P-B improved}.

\begin{remark} \label{rem<1} {\slshape For every $0<\theta<1$ and every $0<\delta<2$, there is $\rho = \rho(\delta, \theta) > 0$ such that for every Banach space $X$, if $x^*\in B_{X^*}$ with $\|x^*\| \leq \theta$, $x\in B_X$ satisfy that $\re x^*(x) > 1-\delta$, then there is a pair $(y,y^*)\in \Pi(X)$ satisfying }
$$
\|x-y\|<\sqrt{2\delta} - \rho \qquad \text{and} \qquad
\|x^*-y^*\|<\sqrt{2\delta} - \rho.
$$
Let us observe that, given $0<\theta<1$, the hypothesis above is not empty only when $1-\theta<\delta$. On the other hand, in the proof it is sufficient to consider only the case of $\delta < 1 + \theta$, because, otherwise, the evident inequality $\re x^*(x) > -\theta = 1 - (1 + \theta)$ implies that there is a pair $(y,y^*)\in \Pi(X)$ satisfying $\|x-y\|<\sqrt{2(1+\theta)}$ and $\|x^*-y^*\|<\sqrt{2(1+\theta)}$, so the statement of our remark holds true with $\rho:= \sqrt{2\delta} - \sqrt{2(1+\theta)}$.
\end{remark}

Next, we rewrite Theorem~\ref{Theorem B-P-B improved} in two
equivalent ways.

\begin{corollary}\label{corollary-BPB-forall}
Let $X$ be a Banach space.
\begin{enumerate}
\item[(a)] Let $0<\eps<2$ and suppose that $x\in B_X$ and $x^*\in
B_{X^*}$ satisfy
    $$
    \re x^*(x)>1-\eps^2/2.
    $$
    Then, there exists $(y,y^*)\in \Pi(X)$ such that
    $$
    \|x-y\|<\eps \quad \text{ and } \quad \|x^*-y^*\|<\eps.
    $$
\item[(b)] Let $0<\delta<2$ and suppose that $x\in B_X$ and
$x^*\in B_{X^*}$ satisfy
    $$
    \re x^*(x)>1-\delta.
    $$
    Then, there exists $(y,y^*)\in \Pi(X)$ such that
    $$
    \|x-y\|<\sqrt{2\delta} \quad \text{ and } \quad \|x^*-y^*\|<\sqrt{2\delta}.
    $$
\end{enumerate}
\end{corollary}

As the last result of this section, we present an example of a
Banach space for which the estimate in Theorem~\ref{Theorem B-P-B
improved} is sharp.

\begin{example}\label{sharp-mod}
{\slshape Let $X$ be the real space $\ell_\infty^{(2)}$. Then,
$\Phi_X^S(\delta)=\Phi_X(\delta) = \sqrt{2\delta}$ for all $\delta \in(0,2)$. }
\end{example}

\begin{proof}
Fix $0<\delta<2$. We consider
$$
z=(1-\sqrt{2\delta},1)\in S_X \quad \text{ and } \quad
z^*=\left(\frac{\sqrt{2\delta}}{2},1-\frac{\sqrt{2\delta}}{2}\right)\in
S_{X^*},
$$
and observe that $z^*(z)=1-\delta$. Now, suppose we may find
$(y,y^*)\in \Pi(X)$ such that $\|z-y\|<\sqrt{2\delta}$ and
$\|z^*-y^*\|<\sqrt{2\delta}$. By the shape of $B_X$, we only have
two possibilities: either $y$ is an extreme point of $B_X$ or
$y^*$ is an extreme point of $B_{X^*}$ (this is actually true for
all two-dimensional real spaces). Suppose first that $y$ is an
extreme point of $B_X$, which has the form $y=(a,b)$ with
$a,b\in\{-1,1\}$. As $$
\|z-y\|=\max\{|1-\sqrt{2\delta}-a|,|1-b|\}<\sqrt{2\delta},
$$
we are forced to have $b=1$ and $a=-1$. Now, we have
$y^*=(-t,1-t)$ for some $0\leq t \leq 1$ and
$$
\|z^*-y^*\|=\frac{\sqrt{2\delta}}{2}+t +
\left|t-\frac{\sqrt{2\delta}}{2}\right|=
\max\left\{\sqrt{2\delta},2t\right\}\geq \sqrt{2\delta},
$$
a contradiction. On the other hand, if $y^*$ is an extreme point
of $B_{X^*}$, then either $y^*=(a,0)$ or $y^*=(0,b)$ for suitable
$a,b\in \{-1,1\}$. In the first case, as
$$
\|z^*-y^*\|=\left|\frac{\sqrt{2\delta}}{2}-a\right| +
1-\frac{\sqrt{2\delta}}{2}<\sqrt{2\delta},
$$
we are forced to have $a=1$ and so, $y=(1,s)$ for suitable $s\in
[-1,1]$. But then $\|z-y\|\geq \sqrt{2\delta}$, which is
impossible. In case $y^*=(0,b)$ with $b=\pm 1$, we have
$$
\|z^*-y^*\|=\frac{\sqrt{2\delta}}{2} +
\left|1-\frac{\sqrt{2\delta}}{2}-b\right|<\sqrt{2\delta},
$$
so $b=-1$ and therefore, $y=(s,-1)$ for suitable $s\in [-1,1]$,
giving $\|z-y\|\geq 2$, a contradiction.
\end{proof}

\section{Basic properties of the moduli}\label{sec:3}

Our first result is the continuity of the
Bishop-Phelps-Bollob\'{a}s moduli.

\begin{prop}\label{Prop-continuidad-1}
Let $X$ be a Banach space. Then, the functions
$$
\delta\longmapsto
\Phi_X(\delta) \qquad \text{and} \qquad \delta\longmapsto
\Phi_X^S(\delta)
$$
are continuous in $(0,2)$.
\end{prop}

We need the following three lemmata which could be of independent
interest.

\begin{lemma}\label{lema-???}
For every pair  $(x_0, x_0^*)\in B_X \times B_{X^*}$  there is a
pair  $(y, y^*) \in \Pi(X)$ with
$$
\re\bigl[ y^*(x_0)  + x_0^*(y)\bigr] \geq 0.
$$
Moreover, if actually $\re x_0^*(x_0) > 0$ then $(y, y^*) \in
\Pi(X)$ can be selected to satisfy
$$
\re \bigl[y^*(x_0)  + x_0^*(y)\bigr] \geq 2 \sqrt{\re \bigl(x_0^*(x_0)\bigr)}.
$$
\end{lemma}

\begin{proof}
1. Take $y_0 \in  S_X \cap \ker x_0^*$ and let $y_0^*$ be a
supporting functional at $y_0$. Then
$$
\re\bigl[y_0^*(x_0)  + x_0^*(y_0)\bigr] = \re y_0^*(x_0)
$$
If the right hand side is positive we can take $y = y_0$,  $y^* =
y_0^*$, in the opposite case take $y = -y_0$,  $y^* = -y_0^*$.

2. Take $y = \frac{x_0}{\|x_0\|}$ and let $y^*$ be a supporting
functional at $y$. Then, since for a fixed $a > 0$ the minimum of
$f(t):= t + \frac{a}{t}$ for $t > 0$ equals $2\sqrt{a}$, we get
\begin{equation*}
\re \bigl[y^*(x_0)  + x_0^*(y)\bigr] = \|x_0\|  +
\frac{1}{\|x_0\|}\re x_0^*(x_0) \geq 2 \sqrt{\re x_0^*(x_0)}.\qedhere
\end{equation*}
\end{proof}

The above lemma allows us to prove the following result which we will use to show the continuity of the Bishop-Phelps-Bollob\'{a}s modulus.

\begin{lemma}\label{lema-continuidad-1}
Let $X$ be a Banach space. Suppose $(x_0,x_0^*)\in A_{X}(\delta_0)$ with $0 <\delta <
\delta_0<2$. Then:
\begin{enumerate}
\item[Case~1:] If  $\delta, \delta_0 \in ]0,1] $ then
$$
\dist\bigl( (x_0,x_{0}^*), A_X(\delta) \bigr) \leq 2 \frac{\sqrt{1
- \delta} - \sqrt{1 - \delta_0}}{1 - \sqrt{1 - \delta_0}}\,.
$$
\item[Case~2:] If  $\delta, \delta_0 \in [1,2) $ then
$$
\dist\bigl( (x_0,x_{0}^*), A_X(\delta) \bigr) \leq 2
\frac{2-\delta_0}{\delta_0} \cdot \frac{\delta_0 -
\delta}{\delta_0 - 1 + \sqrt{1 -  2\delta + \delta \delta_0}}\,.
$$
\end{enumerate}
\end{lemma}

\begin{proof}
Denote  $t = \re x_0^*(x_0)$. Let $(y, y^*) \in \Pi(X)$ be from the
previous lemma (in case 1 we use part 2 of the lemma, in case 2 we
use part 1). For every $\lambda \in [0,1]$ we define $x_\lambda=  (1-\lambda) x_0 + \lambda y$ and $x_\lambda^*=  (1-\lambda) x_0^* + \lambda y^*$. Both $x_\lambda$ and  $x_\lambda^*$ belong to corresponding balls, and $\dist_\infty\left( (x_0, x_0^*),  (x_\lambda, x_\lambda^*)\right) \leq 2 \lambda$.  We have:
\begin{equation} \label{3-eq3}
\re x_\lambda^*(x_\lambda) = (1 - \lambda)^2  t + \lambda (1 -
\lambda) \re\bigl[y^*(x_0)  + x_0^*(y)\bigr] + \lambda^2,
\end{equation}
so in case 1
$$
\re x_\lambda^*(x_\lambda) \geq (1 - \lambda)^2t + 2\lambda (1 -
\lambda)\sqrt{t} + \lambda^2 = \left( (1 - \lambda) \sqrt{t} +
\lambda\right)^2.
$$
Now we are looking for a possibly small value of $\lambda$, for
which $(x_\lambda, x_\lambda^*) \in A_X(\delta)$. If $\delta \geq 1
- t$, the value $\lambda = 0$ is already ok and
$\dist_\infty\left( (x_0, x_0^*), A_X(\delta) \right) = 0$. If $0
< \delta < 1 - t$ then the positive solution in $\lambda$ of the
equation   $ \left( (1 - \lambda) \sqrt{t} + \lambda\right)^2 = 1
- \delta$ is
$$
\lambda_t = \frac{\sqrt{1 - \delta} - \sqrt{t}}{1 - \sqrt{t}}.
$$
Evidently, $\lambda_t \in [0, 1]$, so $(x_{\lambda_t},
x_{\lambda_t}^*) \in A_X(\delta)$. Since  $\lambda_t$ decreases in
t,
$$
\dist_\infty\left( (x_0, x_0^*), A_X(\delta) \right) \leq 2
\lambda_t \leq  2 \lambda_{1 - \delta_0}  = 2 \frac{\sqrt{1 -
\delta} - \sqrt{1 - \delta_0}}{1 - \sqrt{1 - \delta_0}}.
$$
This completes the proof of case 1.

In the case 2 we  may assume $t \leq 1 - \delta$ (otherwise the
corresponding distance is $0$ and the job is done), so   $t \leq
0$. By part 1 of the previous lemma and \eqref{3-eq3}
$$
\re x_\lambda^*(x_\lambda)  \geq  (1 - \lambda)^2t +  \lambda^2,
$$
so we are solving in $\lambda$ the equation
$$
 (1 - \lambda)^2t +  \lambda^2 - 1 + \delta =0,\quad \text{ i.e.} \quad
 (1+t) \lambda^2 -2t\lambda  + (t - 1 + \delta) =0.
$$
The discriminant of this equation is $D = -t \delta -  \delta +
1$. Remark that $D \geq  -(1 - \delta) \delta -  \delta + 1 = (1 -
\delta)^2 \geq 0$ and $t - 1 + \delta\leq 0$, so there is a
positive solution of our equation given by
$$
\lambda_t = \frac{1}{1+t}(t + \sqrt{D}) = \frac{1}{1+t}(t +
\sqrt{1 - t \delta -  \delta}).
$$
This $\lambda_t$ decreases in $t$, so
\begin{equation*}
\lambda_t \leq \lambda_{1 - \delta_0} = \frac{1}{\delta_0}(1 -
\delta_0 + \sqrt{1 -2 \delta +  \delta \delta_0}) =
\frac{2+\delta_0}{\delta_0} \cdot \frac{\delta_0 -
\delta}{\delta_0 - 1 + \sqrt{1 -  2\delta + \delta
\delta_0}}.\qedhere
\end{equation*}
\end{proof}

For the continuity of the spherical modulus, we need the following result.

\begin{lemma}\label{lemma:continuity-spherical}
Let $X$ be a Banach space.
Suppose $(x_0,x_0^*)\in A^S_{X}(\delta_0)$ with $0 <\delta <
\delta_0<2$. Then:
\begin{enumerate}
\item[Case~1:] If  $\delta<1$ then
    $$
    \dist_\infty\bigl( (x_0,x_{0}^*),A^S_X(\delta) \bigr) \leq \frac{4 (\delta_0 -\delta)}{\delta_0}.
    $$
\item[Case~2:] If $\delta\in [1,2)$ and $2-\sqrt{2-\delta_0} <\delta < \delta_0$, then
    $$
    \dist_\infty\bigl( (x_0,x_{0}^*),A^S_X(\delta) \bigr) \leq \frac{2 (\delta_0 -\delta)}{2-\delta}.
    $$
\end{enumerate}
\end{lemma}

\begin{proof}
Let us start with case~1. Fix $\xi\in (0,\delta)$. As $\|x_0^*\|=1$, we may find $y_\xi\in S_X$ satisfying  $x_0^*(y_\xi)> 1 - \xi$. For every $\lambda \in [0,1]$ we define
$$
x(\lambda,\xi)= \lambda x_0 + (1-\lambda)y_\xi.
$$
Consider $\lambda_\xi=\frac{\delta-\xi}{\delta_0-\xi}\in [0,1]$ and write $x_\xi=x(\lambda_\xi,\xi)$. An straightforward verification shows that
$$
\re x_0^*(x_\xi)> 1-\delta
$$
and so, as $1-\delta\geq 0$, we have that $x_\xi\neq 0$ and also that
$$
\re x_0^*\left(\frac{x_{\xi}}{\norm{x_{\xi}}}\right)>1-\delta.
$$
Therefore, $\left(\frac{x_\xi}{\|x_\xi\|}\,,\,x_0^*\right)\in A^S_X(\delta)$. It remains to estimate $\left\|x_0-\frac{x_\xi}{\|x_\xi\|}\right\|$ as follows:
\begin{align*}
\left\|x_0-\frac{x_\xi}{\|x_\xi\|} \right\| \leq \left\|x_0-x_\xi \right\| + \left\|x_\xi-\frac{x_\xi}{\|x_\xi\|} \right\| \leq 2 \left( \frac{\delta_0-\delta}{\delta_0 - \xi} \right) + |\|x_\xi\| -1 |\leq \\ \leq 2 \left( \frac{\delta_0-\delta}{\delta_0 - \xi} \right) + |\|x_\xi\|- \|x_0\| | \leq 2 \left( \frac{\delta_0-\delta}{\delta_0 - \xi} \right) + \|x_\xi - x_0\| \leq 4 \left( \frac{\delta_0-\delta}{\delta_0 - \xi} \right).
\end{align*}
We get the result by just letting $\xi\longrightarrow 0$.

Let us prove case~2. We have to distinguish the values of $\re x_0^*(x_0)$. If $\re x_0^*(x_0)> 1-\delta$, then the proof is done. Suppose otherwise that
$$
1-\delta \geq \re x_0^*(x_0)>1-\delta_0.
$$
Fix $\xi\in \left(0,\min\{2-\delta_0, \frac{4\delta-2-\delta_0-\delta^2}{\delta-1}\}\right)$ (observe that $\frac{4\delta-2-\delta_0-\delta^2}{\delta-1} > 0$ by the conditions on $\delta$). As $\|x_0^*\|=1$, we may find $y_\xi\in S_X$ satisfying  $x_0^*(y_\xi)> 1 - \xi$. Now, we consider
$$
\lambda_\xi=\frac{\delta_0-\delta}{2-\delta-\xi} \qquad \text{and} \qquad x_\xi= x_0 + \lambda_\xi y_\xi.
$$
Notice that $\lambda_\xi \in (0,1)$ (since $\delta < \delta_0$ and $\xi < 2-\delta_0$) and
$$
\norm{x_\xi}\geq \|x_0\|-\lambda\|y_\xi\|=1 - \lambda_\xi>0.
$$
Also, observe that
$$
\re x_0^*(x_\xi)\leq 1-\delta + \lambda_\xi = \frac{(1-\delta)(2-\delta-\xi)+ \delta_0-\delta}{2-\delta-\xi}
$$
so, $\re x_0^*(x_\xi)\leq 0$ since $\xi \leq \frac{4\delta-2-\delta_0-\delta^2}{\delta-1}$. Now,
$$
\re x_0^*\left(\frac{ x_\xi}{\norm{x_\xi}}\right)\geq \re x_0^*\left(\frac{ x_\xi}{1-\lambda_\xi}\right)> \frac{1-\delta_0 + \lambda_\xi(1-\xi)}{1-\lambda_\xi}= 1-\delta.
$$
Therefore, $\left(\frac{x_\xi}{\|x_\xi\|}\,,\,x_0^*\right)\in A^S_X(\delta)$. It remains to estimate $\left\|x_0-\frac{x_\xi}{\|x_\xi\|}\right\|$ as follows:
\begin{align*}
\left\|x_0-\frac{x_\xi}{\|x_\xi\|} \right\| &\leq \left\|x_0-x_\xi \right\| + \left\|x_\xi-\frac{x_\xi}{\|x_\xi\|} \right\| \leq   \frac{\delta_0-\delta}{2-\delta -\xi}  + \bigl|\|x_\xi\| -1 \bigr|\leq \\ &\leq \frac{\delta_0-\delta}{2-\delta -\xi} + \bigl|\|x_\xi\|- \|x_0\| \bigr| \leq \frac{\delta_0-\delta}{2-\delta -\xi} + \|x_\xi - x_0\| \leq 2 \left( \frac{\delta_0-\delta}{2-\delta -\xi} \right).
\end{align*}
Consequently, letting $\xi\longrightarrow 0$, we get
\begin{equation*}
\dist_\infty\bigl( (x_0,x_{0}^*),A^S_X(\delta) \bigr) \leq \frac{2 (\delta_0 -\delta)}{2-\delta}.\qedhere
\end{equation*}

\end{proof}

\begin{proof}[Proof of Proposition~\ref{Prop-continuidad-1}]
Let us give the proof for $\Phi_X(\delta)$. Observe that for $\delta_1,\delta_2\in (0,2)$ with
$\delta_1<\delta_2$, one has
\begin{align*}
0<\Phi_X(\delta_2) - \Phi_X(\delta_1)  =
d_H\left(A_X(\delta_2),\Pi(X)\right) -
d_H\left(A_X(\delta_1),\Pi(X)\right) \leq
d_H\left(A_X(\delta_2),A_X(\delta_1) \right).
\end{align*}
Now, the continuity follows routinely from
Lemma~\ref{lema-continuidad-1}.

An analogous argument allows to prove the continuity of $\Phi_X^S(\delta)$ from Lemma~\ref{lemma:continuity-spherical}.
\end{proof}

%

The following lemma will be used to show that the approximation in
the space is not worse than the approximation in the dual. It is
actually an easy application of the Principle of Local
Reflexivity.

\begin{lemma}\label{lema-modulo-dual}
For $\eps>0$, let $(x,x^*)\in B_X\times B_{X^*}$ and let
$(\tilde{y}^*,\tilde{y}^{**})\in \Pi(Y^*)$ such that
$$
\|x^*-\tilde{y}^*\|<\eps \quad \text{and} \quad
\|x-\tilde{y}^{**}\|<\eps.
$$
Then there is a pair $(y,y^*)\in \Pi(X)$ such that
$$
\|x-y\|  <\eps \quad \text{and} \quad \|x^*- y^*\|<\eps.
$$
\end{lemma}

\begin{proof}
First chose $\eps'<\eps$ such that still
$$
\|x^*-\tilde{y}^*\|<\eps' \quad \text{and} \quad
\|x-\tilde{y}^{**}\|<\eps'.
$$
Now, we consider $\xi>0$ such that
$$
(1+\xi)\eps' + \xi + \sqrt{\frac{2\xi}{1+\xi}} <\eps,
$$
and use the Principle of Local Reflexivity (see
\cite[Theorem~11.2.4]{Albiac-Kalton}, for instance) to get an
operator $T:\lin \{x,\tilde{y}^{**}\} \longrightarrow X$
satisfying
$$
\|T\|,\|T^{-1}\|\leq 1+\xi,\quad T(x)=x,\quad
\tilde{y}^*(T(\tilde{y}^{**}))=y^{**}(\tilde{y}^*)=1.
$$
Next, we consider
$\tilde{x}=\dfrac{T(\tilde{y}^{**})}{\|T(\tilde{y}^{**})\|}\in
S_X$ and $\tilde{x}^*=\tilde{y}^*\in S_{X^*}$, observe that
$$
\re \tilde{x}^*(\tilde{x})>\frac{1}{1+\xi}=1 - \frac{\xi}{1+\xi},
$$
and we use Corollary~\ref{corollary-BPB-forall} to get $(y,y^*)\in
\Pi(X)$ satisfying that
$$
\|\tilde{x}-y\|<\sqrt{\frac{2\xi}{1+\xi}} \quad \text{and} \quad
\|\tilde{x}^*-y^*\|<\sqrt{\frac{2\xi}{1+\xi}}.
$$
Let us show that $(y,y^*)\in \Pi(X)$ fulfill our requirements:
\begin{align*}
\|x-y\| &\leq \|T(x)-T(\tilde{y}^{**})\| + \|T(\tilde{y}^{**}) -\tilde{x}\| + \|\tilde{x}-y\| \\
 &< (1+\xi)\eps' + \xi + \sqrt{\frac{2\xi}{1+\xi}} <\eps
\intertext{and, analogously,}
\|x^*-y^*\| &\leq \|x^*- \tilde{y}^*\| + \|\tilde{y}^*-y^*\| <
\eps' + \sqrt{\frac{2\xi}{1+\xi}} <\eps.\qedhere
\end{align*}
\end{proof}

\begin{prop}\label{proposicion-modulo-dual}
Let $X$ be a Banach space. Then
$$
\Phi_X(\delta)\leq \Phi_{X^*}(\delta) \quad \text{and} \quad \Phi^S_X(\delta)\leq \Phi^S_{X^*}(\delta)
$$
for every $\delta \in (0,2)$.
\end{prop}

\begin{proof}
The proof is the same for both moduli, so we are only giving the case of $\Phi_X(\delta)$. Fix $\delta\in (0,2)$. We consider any $\eps>0$ such that
$\Phi_{X^*}(\delta)<\eps$ and for a given $(x,x^*)\in A_X(\delta)$
consider $(x^*,x)\in A_{X^*}(\delta)$ (we identify $X$ as a
subspace of $X^{**}$) and so we may find
$(\tilde{y}^*,\tilde{y}^{**})\in \Pi(Y^*)$ such that
$$
\|x^*-\tilde{y}^*\|<\eps \quad \text{and} \quad
\|x-\tilde{y}^{**}\|<\eps.
$$
Now, an application of the previous lemma gives us a $(y,y^*)\in
\Pi(X)$ such that
$$
\|x-y\|   < \eps \quad \text{and} \|x^*- y^*\|<\eps.
$$
This means that $\Phi_X(\delta)\leq \eps$ and, therefore,
$\Phi_X(\delta)\leq \Phi_{X^*}(\delta)$, as desired.
\end{proof}

We do not know whether the inequalities in
Proposition \ref{proposicion-modulo-dual} can be strict. Of
course, this can not be the case when the space is reflexive.

\begin{corollary}
For every reflexive Banach space $X$, one has
$\Phi_X(\delta)=\Phi_{X^*}(\delta)$ and $\Phi_X^S(\delta)=\Phi^S_{X^*}(\delta)$ for every $0<\delta<2$.
\end{corollary}

Our last result in this section states that when the
Bishop-Phelps-Bollob\'{a}s modulus is the worst possible, then the
spherical Bishop-Phelps-Bollob\'{a}s modulus is also the worst
possible.

\begin{prop} \label{prop-Phi-Phi^s}
Let $X$ be a Banach space. For every $\delta \in (0, 2)$, the condition $\Phi_X(\delta) = \sqrt{2 \delta}$ is equivalent to the condition $\Phi_X^S(\delta) = \sqrt{2 \delta}$.
\end{prop}

\begin{proof}
Since $\Phi_X^S(\delta) \leq \Phi_X(\delta) \leq \sqrt{2 \delta}$,
the implication $\bigl[\Phi_X^S(\delta) = \sqrt{2
\delta}\bigr]$ $\Rightarrow$ $\bigl[\Phi_X(\delta) = \sqrt{2
\delta}\bigr]$
is evident. Let us prove the inverse
implication. Let $\Phi_X(\delta) = \sqrt{2 \delta}$. Then there is
a sequence of pairs $(x_n, x_n^*) \in B_X \times B_{X^*}$ such
that $\re  x_n^*(x_n) > 1 - \delta$ but  for every $(y,y^*)\in
\Pi(X)$ we have
$$
\|x_n-y\|  \geq \sqrt{2 \delta} - \frac1n \quad \text{or  }
\|x_n^*- y^*\| \geq \sqrt{2 \delta} - \frac1n.
$$
An application of Remark~\ref{rem<1} gives us that $\|x_n^*\| \longrightarrow
1$ as $n \to \infty$. As the duality argument given in
Lemma~\ref{lema-modulo-dual} implies the dual version of
Remark~\ref{rem<1}, we also have $\|x_n\| \longrightarrow 1$ as $n \to \infty$. Denote $\tilde{x_n} = \frac{x_n}{\|x_n\|}$,
$\tilde{x}_n^* = \frac{x_n^*}{\|x_n^*\|}$. In the case when
$\delta \in (0, 1]$, we have $\re  \tilde{x}_n^*(\tilde{x}_n) > 1
- \delta$ but for every $(y,y^*)\in \Pi(X)$
$$
\|\tilde{x}_n-y\|  \geq \sqrt{2 \delta} - \frac1n -
\|x_n-\tilde{x}_n\| \quad \text{or  } \|\tilde{x}_n^*- y^*\| \geq
\sqrt{2 \delta} - \frac1n - \|\tilde{x}_n^*- x_n^*\|.
$$
Since the right-hand sides of the above inequalities go to
$\sqrt{2 \delta}$, we get the condition $\Phi_X^S(\delta) =
\sqrt{2 \delta}$.

In the case of $\delta \in (1, 2)$, we no longer know that $\re
\tilde{x}_n^*(\tilde{x}_n) > 1 - \delta$, but what we do know is
that $\liminf \re  \tilde{x}_n^*(\tilde{x}_n) \geq 1 - \delta$, and
that gives us the desired condition $\Phi_X^S(\delta) = \sqrt{2
\delta}$ thanks to the continuity of the spherical modulus (Proposition~\ref{Prop-continuidad-1}).
\end{proof}

\section{Examples}\label{sec:4}

We start with the simplest example of $X=\R$.

\begin{example}
$\Phi_{\R}(\delta)=\begin{cases} \delta & \text{if $0<\delta\leq
1$} \\ \sqrt{\delta-1}+1 & \text{if $1<\delta<2$} \end{cases}$,
\qquad $\Phi_{\R}^S(\delta)=0$ for every $\delta\in(0,2)$.
\end{example}

\begin{proof}
We first fix $\delta\in (0,1]$. First observe that taking
$x=1-\delta$, $x^*=1$, it is evident that $\Phi_{\R}(\delta)\geq
\delta$. For the other inequality, we fix $x,x^*\in [-1,1]$ with
$x^* x> 1-\delta$. Then, $x$ and $x^*$ have the same sign and we
have that $|x|>1-\delta$ and $|x^*|>1-\delta$. Indeed, if
$|x|<1-\delta$, as $|x^*|\leq 1$, one has $x^*x=|x^*x|<1-\delta$,
a contradiction; the other inequality follows in the same manner.
Finally, one deduces that $|x-\mathrm{sign}(x)|<\delta$ and
$|x^*-\mathrm{sign}(x^*)|<\delta$, as desired.

Second, fix $\delta\in (1,2)$. On the one hand, taking
$x=\sqrt{\delta-1}$, $x^*=-\sqrt{\delta-1}$, one has
$x^*x=1-\delta$. As $|x+1|=\sqrt{\delta-1}+1$ and
$|x^*-1|=\sqrt{\delta-1}+1$, it follows that $\Phi_\R(\delta)\geq
\sqrt{\delta-1}+1$. For the other inequality, we fix $x,x^*\in
[-1,1]$ with $x^*x>1-\delta$. If $x$ and $x^*$ have the same sign,
which we may and do suppose positive, then $|x-1|\leq 1<\delta$
and $|x^*-1|\leq 1< \delta$ and the same is true if one of them is
null. Therefore, to prove the last case we may and do suppose that
$x>0$ and $x^*<0$. Now, if we suppose, for the sake of
contradiction, that
$$
|x-(-1)|\geq \sqrt{\delta-1} + 1 \quad \text{and} \quad
|x^*-1|\geq \sqrt{\delta-1} + 1,
$$
we get $x\geq \sqrt{\delta-1}$ and $-x^*\geq \sqrt{\delta-1}$, so
$-x^*x\geq \delta-1$ or, equivalently, $x^*x\leq 1-\delta$, a
contradiction. Therefore, either $|x-(-1)|<\sqrt{\delta-1}+1$ and
$|x^*-(-1)|<1<\sqrt{\delta-1}+1$ or $|x^*-1|<\sqrt{\delta-1}+1$
and $|x-1|<1<\sqrt{\delta-1}+1$.

The result for $\Phi_\R^S$ is an obvious consequence of the fact
that $S_\R=\{-1,1\}$.
\end{proof}

Let us observe that the above proof gives actually a lower bound
for $\Phi_X(\delta)$ for every Banach space $X$ when
$\delta\in(0,1]$.

\begin{remark}\label{remark-lowerbound-1}
{\slshape Let $X$ be a Banach space. Then $\Phi_X(\delta)\geq
\delta$ for every $\delta\in (0,1]$.}\ Indeed, consider $x_0\in
S_X$ and $x_0^*\in S_{X^*}$ with $x^*_0(x_0)=1$ and write
$x=(1-\delta)x_0$ and $x^*=x_0^*$. Then $\re x^*(x)=1-\delta$ and
$\dist(x,S_X)=\delta$.
\end{remark}

We do not know a result giving a lower bound for $\Phi_X(\delta)$
when $\delta>1$, outside of the trivial one $\Phi_X(\delta) \geq
1$. Also, we do not know if the lower bound for the behavior of $\Phi_X(\delta)$ in a neighborhood of $0$
given in the remark above can be improved for Banach spaces of
dimension greater than or equal to two.

We next calculate the moduli of a Hilbert space of (real) dimension greater than one.

\begin{example} Let $H$ be a Hilbert space of dimension over $\R$ greater than or equal to two. Then:
\begin{enumerate}
\item[(a)] $\Phi^S_H(\delta)=\sqrt{2-\sqrt{4-2\delta}}$ for every $\delta\in (0,2)$.
\item[(b)] For $\delta\in (0,1]$, $\Phi_H(\delta)= \max \left\{ \delta,\sqrt{2-\sqrt{4-2\delta}}\right\}$. For $\delta\in (1,2)$, $\Phi_H(\delta)=\sqrt{\delta}$.
\end{enumerate}
\end{example}

\begin{proof} As we commented in the introduction, both $\Phi_H$ and $\Phi_H^S$ only depend on the real structure of the space, so we may and do suppose that $H$ is a real Hilbert space of dimension greater than or equal to $2$. Let us also recall that $H^*$ identifies with $H$ and that the action of a vector $y\in H$ on a vector $x\in H$ is nothing but their inner product denoted by $\langle x,y\rangle$. In particular,
$$
\Pi(H)=\bigl\{(z,z)\in S_H\times S_H \bigr\}.
$$
Therefore, for every $\delta\in (0,2)$, $\Phi_H(\delta)$ (resp.\ $\Phi_H^S(\delta)$) is the infimum of those $\eps>0$ such that whenever $x,y\in B_H$ (resp.\ $x,y\in S_H$) satisfies $\langle x,y\rangle \geq 1-\delta$, there is $z\in S_H$ such that $\|x-z\|\leq\eps$ and $\|y-z\|\leq\eps$.

We will use the following (easy) claim in both the proofs of (a) and (b).

\emph{Claim:} Given $x,y\in S_H$ with $x+y\neq 0$, write $z=\frac{x+y}{\|x+y\|}$ to denote the normalized midpoint. Then
$$
\|x-z\|=\|y-z\|=\sqrt{2-\sqrt{2 + 2\langle x,y\rangle}}.
$$
Indeed, we have $\|x-z\|^2=2-2\langle x,z\rangle$ and
$$
2\langle x,z\rangle = \frac{2\langle x,x+y\rangle}{\|x+y\|}=\frac{2+2\langle x,y\rangle}{\sqrt{2+2\langle x,y\rangle}},
$$
giving $\|x-z\|=\sqrt{2-\sqrt{2 + 2\langle x,y\rangle}}$, being the other equality true by symmetry.

(a). Let first prove that $\Phi_H^S(\delta)\leq \sqrt{2-\sqrt{4-2\delta}}$. Take $x,y\in S_H$ with $\langle x,y\rangle \geq 1-\delta$ (so $x+y\neq 0$), consider $z=\frac{x+y}{\|x+y\|}\in S_H$ and use the claim to get that
$$
\|x-z\|=\|y-z\|=\sqrt{2-\sqrt{2 + 2\langle x,y\rangle}}\leq \sqrt{2-\sqrt{4-2\delta}}.
$$
To get the other inequality, we fix an ortonormal basis $\{e_1,e_2,\ldots\}$ of $H$, consider
$$
x=\sqrt{1-\delta/2}\,e_1 + \sqrt{\delta/2}e_2\in S_H \quad \text{and} \quad y=\sqrt{1-\delta/2}\,e_1 - \sqrt{\delta/2}e_2\in S_H
$$
and observe that $\langle x,y\rangle =1-\delta$. Now, given $z\in S_H$, we write $z_1=\langle z,e_1\rangle$, $z_2=\langle z,e_2\rangle$, and observe that
\begin{align*}
\max\{\|z-x\|^2,\|z-y\|^2\}&=\max_\pm\left\{|z_1-\sqrt{1-\delta/2}|^2 + |z_2 \pm \sqrt{\delta/2}|^2 + 1 - z_1^2-z_2^2\right\} \\ & = z_1^2 + 1 - \delta/2 - 2z_1\sqrt{1-\delta/2} + \max_\pm |z_2 \pm \sqrt{\delta/2}|^2 + 1 - z_1^2-z_2^2 \\ & = 2 -2z_1\sqrt{1-\delta/2} +2|z_2|\sqrt{\delta/2} \geq 2 - 2\sqrt{1-\delta/2}.
\end{align*}
It follows that $\Phi_H^S(\delta)\geq \sqrt{2-\sqrt{4-2\delta}}$, as desired.

(b). We first fix $\delta\in (0,1)$ and write $\varepsilon_0=\max \left\{ \delta,\sqrt{2-\sqrt{4-2\delta}}\right\}$. The inequality $\Phi_H(\delta)\geq \eps_0$ follows by Remark~\ref{remark-lowerbound-1}, the fact that $\Phi_H(\delta)\geq \Phi_H^S(\delta)$ and the result in item (a). To get the other inequality, we first observe that
\begin{equation}\label{eq:Hilbert--1}
\Phi_H(\delta)\leq \Phi_{\lin\lbrace x, y \rbrace}(\delta)\hspace{5mm} \forall x,y\in B_H \hspace{3mm}\mbox{with}\hspace{3mm} \langle x,y\rangle = 1-\delta.
\end{equation}
This follows from the obvious fact that $\Phi_{\cdot}(\delta)$ increases when we restrict to subspaces. So, we are done if we restrict to the two-dimensional case and consider two points $P=(\norm{P},0)$, $Q=(q_1,q_2)$ with $q_2\geq 0$ and $\|P\|\geq \|Q\|$, satisfying $\langle P,Q \rangle \geq 1-\delta$, and we find $z\in S_H$ such that $\|P-z\|\leq \eps_0$ and $\|Q-z\|\leq \eps_0$. Now, it is straightforward to check that we have $\norm{P}\in \left[\sqrt{1-\delta},1 \right]$, and $q_1=\frac{1-\delta}{\norm{P}}\in \left[1-\delta,\sqrt{1-\delta} \right]$. Figure~\ref{figure:Hilbert} helps to the better understanding of the rest of the proof.

\definecolor{xdxdff}{rgb}{0.49,0.49,1}
\definecolor{uuuuuu}{rgb}{0.27,0.27,0.27}
\begin{figure}
\begin{tikzpicture}[line cap=round,line join=round,>=triangle 45,x=8.0cm,y=8.0cm]
\draw[->,color=black] (0,0) -- (1.25,0);
\foreach \x in {-0.2,0.2,0.3,0.4,0.5,0.6,0.7,0.8,0.9,1,1.1,1.2,1.3,1.4,1.5,1.6,1.7,1.8}
\draw[shift={(\x,0)},color=black] (0pt,-2pt);
\draw[->,color=black] (0,0) -- (0,1.25);
\clip(-0.26,-0.21) rectangle (1.9,1.25);
\draw [shift={(0,0)}] plot[domain=0:1.57,variable=\t]({1*1*cos(\t r)+0*1*sin(\t r)},{0*1*cos(\t r)+1*1*sin(\t r)});
\draw [dash pattern=on 1pt off 1pt on 1pt off 4pt] (0.53,0.85)-- (0.53,0);
\draw [dotted] (0.53,0.85)-- (0.87,0.49);
\draw [dotted] (0.87,0.49)-- (1,0);
\draw [dash pattern=on 1pt off 1pt on 1pt off 4pt] (0.6,0.8)-- (0.6,0);
\draw [dotted] (0.87,0.49)-- (0.6,0.49);
\begin{scriptsize}
\draw [color=black] (1,0)-- ++(-1.5pt,0 pt) -- ++(3.0pt,0 pt) ++(-1.5pt,-1.5pt) -- ++(0 pt,3.0pt);
\draw[color=black] (1,-0.06) node {$A$};
\draw [color=black] (0.46,0)-- ++(-1.5pt,0 pt) -- ++(3.0pt,0 pt) ++(-1.5pt,-1.5pt) -- ++(0 pt,3.0pt);
\draw[color=black] (0.43,-0.06) node {$1-\delta$};
\draw [color=black] (0.68,0)-- ++(-1.5pt,0 pt) -- ++(3.0pt,0 pt) ++(-1.5pt,-1.5pt) -- ++(0 pt,3.0pt);
\draw[color=black] (0.70,-0.06) node {$\sqrt{1-\delta}$};
\draw [color=black] (0.86,0)-- ++(-1.5pt,0 pt) -- ++(3.0pt,0 pt) ++(-1.5pt,-1.5pt) -- ++(0 pt,3.0pt);
\draw[color=black] (0.86,-0.06) node {$P$};
\draw [color=black] (0.53,0)-- ++(-1.5pt,0 pt) -- ++(3.0pt,0 pt) ++(-1.5pt,-1.5pt) -- ++(0 pt,3.0pt);
\draw[color=black] (0.52,-0.06) node {$\frac{1-\delta}{\|P\|}$};
\fill [color=black] (0.53,0.85) circle (1.5pt);
\draw[color=black] (0.54,0.89) node {$B$};
\fill [color=black] (0.87,0.49) circle (1.5pt);
\draw[color=black] (0.91,0.51) node {$M$};
\draw [color=black] (0.6,0)-- ++(-1.5pt,0 pt) -- ++(3.0pt,0 pt) ++(-1.5pt,-1.5pt) -- ++(0 pt,3.0pt);
\draw[color=uuuuuu] (0.60,-0.06) node {$q_1$};
\fill [color=uuuuuu] (0.6,0.8) circle (1.5pt);
\draw[color=uuuuuu] (0.61,0.85) node {$D$};
\fill [color=uuuuuu] (0.6,0.49) circle (1.5pt);
\draw[color=uuuuuu] (0.64,0.52) node {$C$};
\end{scriptsize}
\end{tikzpicture}
\caption{Calculating $\Phi_H(\delta)$ for $\delta\in (0,1)$}
\label{figure:Hilbert}
\end{figure}
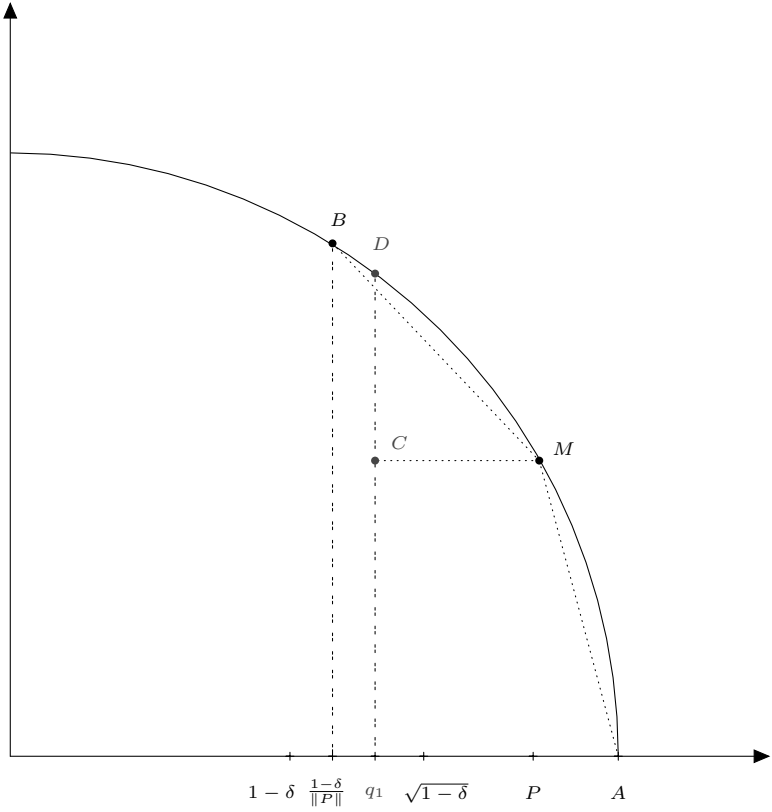

Consider $M=\left(\sqrt{\frac{1-\delta+\norm{P}}{2\norm{P}}}, \sqrt{\frac{\norm{P}-(1-\delta)}{2\norm{P}}}\right)$,
which is the normalized midpoint between $A=(1,0)$ and $B=\big(\frac{1-\delta}{\norm{P}}, \sqrt{1-(\frac{1-\delta}{\norm{P}})^2}\big)$ and write $\Delta$ to denote the arc of the unit sphere of $H$ between $A$ and $M$. We claim that $Q\in \bigcup_{z\in \Delta} B(z,\eps_0)$ and $P\in \bigcap_{z\in \Delta} B(z,\eps_0)$. Observe that this gives that there is $z\in \Delta\subset S_H$ whose distance to $P$ and $Q$ is less than or equal to $\eps_0$, finishing the proof. Let us prove the claim. First, we show that $Q=(q_1,q_2)\in \bigcup_{z\in \Delta} B(z,\eps_0)$. If $q_2\leq \sqrt{\frac{\norm{P}-(1-\delta)}{2\norm{P}}}$, the ball of radius $\varepsilon_0$ centered in the point of $\Delta$ with second coordinate equal to $q_2$ contains the point $Q$ since $\varepsilon_0 \geq \dist\left((q_1,0),A\right)\geq \dist(Q,\Delta)$. For greater values of $q_2$, write first $C=\left(q_1,\sqrt{\frac{\norm{P}-(1-\delta)}{2\norm{P}}}\right)$, which belongs to $B(M,\eps_0)$ by the previous argument. Also, as $M$ is the normalized mid point between $A$ and $B$, we have by the claim at the beginning of this proof that
$$
\|M-B\|= \sqrt{2-\sqrt{2 + 2\langle A,B\rangle}}=\sqrt{2-\sqrt{2 + 2\frac{1-\delta}{\|P\|}}}  \leq \sqrt{2-\sqrt{4-2\delta}}\leq \eps_0
$$
so, also, $\|M-D\|\leq \eps_0$. Therefore, both the points $C$ and $C$ belong to $B(M,\eps_0)$, so also the whole segment $[C,D]$ is contained there, and this proves the first part of the claim.  To show the second part of the claim, that $P\in \bigcap_{z\in\Delta} B(z,\eps_0)$, we consider the function
$$
f(p):= 1+p^2-\sqrt{2p(p+1-\delta)} \qquad \bigl(p\in[\sqrt{1-\delta},1]\bigr)
$$
and observe that it is a convex function, so
$$
f(p)\leq \max\{f(1),f(\sqrt{1-\delta}\}\leq \varepsilon_0^2.
$$
It follows that
$$
\|P-M\|=\sqrt{1+\norm{P}^2-\sqrt{2\norm{P}(\norm{P}+1-\delta)}}\leq \varepsilon_0,
$$
hence $M\in B(P,\varepsilon_0)$. As also $A\in B(P,\eps_0)$, it follows that the whole circular arc $\Delta$ is contained in $B(P,\eps_0)$ or, equivalently, that $P\in \bigcap_{z\in \Delta} B(z,\varepsilon_0)$.

Let now fix $\delta \in (1,2)$. Analogously to what we did before in equation~\eqref{eq:Hilbert--1}, to show that $\Phi_H(\delta)\leq \sqrt{\delta}$, it is enough to consider the two-dimensional case and that, given
$p=(\norm{p},0)\in B_H$, $q=(q_1, q_2)\in B_H$ with
$q_2\geq 0$, to find $z\in S_H$ such that $\|z-P\|,\|z-Q\|\leq \sqrt{\delta}$. Routine computations show that
$$
z =
\left(\frac{\norm{p}+q_1}{2},\sqrt{1-\big(\frac{\norm{p}+q_1}{2}\big)^2}\right)\in
S_H
$$
does the job. For the other inequality, we fix an ortonormal basis $\{e_1,e_2,\ldots\}$ of $H$, consider
$$
P=\sqrt{\delta-1}\,e_1\in B_H,\qquad Q=-\sqrt{\delta-1}\,e_1\in B_H
$$
and observe that $\langle P,Q\rangle =1-\delta$. For any $z\in S_H$, we write $z_1=\langle z,e_1\rangle$ and we compute
\begin{align*}
\max\{\|z-P\|^2,\|z-Q\|^2\}&=\max\left\{|z_1-\sqrt{\delta-1}|^2 + 1-|z_1|^2\,,\, |z_1+\sqrt{\delta-1}|^2 + 1-|z_1|^2\right\} \\ & = \max_{\pm}{|z_1 \pm \sqrt{\delta-1}|^2} +1 - |z_1|^2 = (|z_1|+\sqrt{\delta-1})^2 + 1-|z_1|^2 \\ & = \delta + 2\sqrt{\delta-1}|z_1| \geq \delta.
\end{align*}
It follows that $\Phi_H(\delta)\geq \sqrt{\delta}$, as desired.
\end{proof}

Our next aim is to present a number of examples for which the
values of the Bishop-Phelps-Bollob\'{a}s moduli are the maximum
possible, namely $\Phi_X^S(\delta)=\Phi_X(\delta)=\sqrt{2\delta}$ for small $\delta$'s. As we always have $\Phi_X^S(\delta)\leq \Phi_X(\delta)\leq \sqrt{2\delta}$, it is enough if we prove
the formally stronger result that $\Phi_X^S(\delta)=\sqrt{2\delta}$ for small $\delta$'s (actually, the two facts are equivalent, see Proposition~\ref{prop-Phi-Phi^s}), and this is what we will show. It happens that all of the examples have in common that they contains an isometric copy of the real space $\ell_\infty^{(2)}$ or $\ell_1^{(2)}$. In the next section we will show that the latter is a necessary condition that it is not actually sufficient.

The first result is about Banach spaces admitting an
$L$-descomposition. As a consequence we will calculate the moduli
of $L_1(\mu)$ spaces.

\begin{prop}\label{prop-ell-1-sum}
Let $X$ be a Banach space. Suppose that there are two
(non-trivial) subspaces $Y$ and $Z$ such that $X=Y\! \oplus_1
\!Z$. Then $\Phi_X(\delta)=\Phi_X^S(\delta)=\sqrt{2\delta}$ for every
$\delta\in(0,1/2]$.
\end{prop}

\begin{proof}
Fix $\delta\in (0,1/2]$ and consider $(y_0,y_0^*)\in \Pi(Y)$ and
$(z_0,z_0^*)\in \Pi(Z)$ and write
\begin{equation*}
x_0=\left(\frac{\sqrt{2\delta}}{2}\,y_0\,,\,\Bigl(1-\frac{\sqrt{2\delta}}{2}\Bigr)\,z_0\right)
\in S_X \qquad x_0^*=\Bigl(\bigl(1-\sqrt{2\delta}\bigr) y_0^*\,,\,
z_0^*\Bigr)\in S_{X^*}.
\end{equation*}
It is clear that $\re x_0^*(x_0)=1-\delta$. Now, suppose that we
may choose $(x,x^*)\in \Pi(X)$ such that
$$
\|x_0-x\|<\sqrt{2\delta} \quad \text{ and } \quad
\|x_0^*-x^*\|<\sqrt{2\delta}.
$$
Write $x=(y,z)\in Y\oplus_1 Z$, $x^*=(y^*,z^*)\in Y^*\oplus_\infty
Z^*$ and observe that
$$
1=\re x^*(x)=\re y^*(y) + \re z^*(z)\leq \|y^*\|\|y\| +
\|z^*\|\|z\|\leq \|y\|+\|z\|=1,
$$
therefore, we have
\begin{equation}\label{eq:ell1-sum}
\re y^*(y)=\|y^*\|\|y\|.
\end{equation}
Now, we have
$$
\left|\bigl(1-\sqrt{2\delta}\bigr) - \|y^*\| \right|\leq
\left\|\bigl(1-\sqrt{2\delta}\bigr) y_0^* \,-\,
y^*\right\|<\sqrt{2\delta}
$$
from which follows that $\|y^*\|<1$ and so, $y=0$ by
\eqref{eq:ell1-sum}, giving $\|z\|=\|x\|=1$. But then,
\begin{align*}
\|x_0-x\| = \left\|\frac{\sqrt{2\delta}}{2}\,y_0\right\| +
\left\|\Bigl(1-\frac{\sqrt{2\delta}}{2}\Bigr)z_0\, -\, z\right\|
\geq \frac{\sqrt{2\delta}}{2} +
\left|\Bigl(1-\frac{\sqrt{2\delta}}{2}\Bigr) \,-\, \|z\|\right|=
\sqrt{2\delta},
\end{align*}
a contradiction. We have proved that $\Phi_X(\delta)\geq
\sqrt{2\delta}$, being the other inequality always true.
\end{proof}

The result above produces the following example.

\begin{example}\label{example:L_1(mu)}
{\slshape Let $(\Omega,\Sigma,\mu)$ be a measure space such that
$L_1(\mu)$ has dimension greater than one and let $E$ be any non-zero Banach space. Then, $\Phi_{L_1(\mu,E)}(\delta) = \Phi_{L_1(\mu,E)}^S(\delta)=\sqrt{2\delta}$
for every $\delta\in(0,1/2]$.} \newline Indeed, we may find two
measurable sets $A,B\subset \Omega$ with empty intersection such that $\Omega=A\cup B$. Then
$Y=L_1(\mu|_A,E)$ and $Z=L_1(\mu|_B,E)$ are non-null, $L_1(\mu,E)=Y\oplus_1 Z$ and so the
results follows from Proposition~\ref{prop-ell-1-sum}.
\end{example}

Particular case of the above example are $\ell_1$ and $L_1[0,1]$.

It is immediate that with a dual argument than the one given in
Proposition~\ref{prop-ell-1-sum} it is possible to deduce the same
for a Banach space which decomposes as an $\ell_\infty$-sum.
Actually, in this case we will get a better result using ideals
instead of subspaces.

\begin{prop}\label{prop-dual-ell-1-sum}
Let $X$ be a Banach space. Suppose that $X^*=Y\! \oplus_1\!Z$
where $Y$ and $Z$ are (non-trivial) subspaces of $X^*$ such that
$\overline{Y}^{w^*}\!\neq X^*$ and $\overline{Z}^{w^*}\!\neq X^*$
($w^*$ is the weak$^*$-topology $\sigma(X^*,X)$). Then
$\Phi_X(\delta)=\Phi_X^S(\delta)=\sqrt{2\delta}$ for every $\delta\in(0,1/2]$.
\end{prop}

\begin{proof}
We claim that there are $y_0,z_0\in S_X$ and $y_0^*\in S_{Y}$ and
$z_0^*\in S_Z$ such that
\begin{equation*}
\re y_0^*(y_0)=1,\quad \re z_0^*(z_0)=1,\quad y^*(z_0)=0 \ \forall
y^*\in Y,\quad z^*(y_0)=0 \ \forall z^*\in Z.
\end{equation*}
Indeed, we define $y_0$ and $y_0^*$, being $z_0$ and $z_0^*$
analogous. By assumption there is $y_0\in S_X$ such that
$z^*(y_0)=0$ for every $z^*\in Z$ and we may choose $x^*\in
S_{X^*}$ such that $\re x^*(y_0)=1$ and we only have to prove that
$x^*\in Y$ and then write $y_0^*=x^*$. But we have $x^*=y^*+z^*$
with $y^*\in Y$, $z^*\in Z$ and
$$
1=\re x^*(y_0)=\re y^*(y_0) \leq \|y^*\|\leq \|y^*\| + \|z^*\|=1,
$$
so $z^*=0$ and $x^*\in Y$.

We now define
\begin{equation*}
x^*_0=\left(\frac{\sqrt{2\delta}}{2}\,y^*_0\,,
\,\Bigl(1-\frac{\sqrt{2\delta}}{2}\Bigr)\,z^*_0\right) \in S_{X^*}
\qquad x_0=\bigl(1-\sqrt{2\delta}\bigr) y_0 + z_0\in X
\end{equation*}
and first observe that $\|x_0\|\leq 1$; indeed, for every
$x^*=y^*+z^*\in S_{X^*}$ one has
$$
|x^*(x_0)|=\left|\bigl(1-\sqrt{2\delta}\bigr) y^*(y_0) +
z^*(z_0)\right|\leq \bigl(1-\sqrt{2\delta}\bigr)\|y^*\| +
\|z^*\|\leq \|y^*\|+\|z^*\|=1.
$$
It is clear that $\re x_0^*(x_0)=1-\delta$. Now, suppose that we
may choose $(x,x^*)\in \Pi(X)$ such that
$$
\|x_0-x\|<\sqrt{2\delta} \quad \text{ and } \quad
\|x_0^*-x^*\|<\sqrt{2\delta}.
$$
We consider the semi-norm $\|\cdot\|_Y$ defined on $X$ by
$\|x\|_Y:=\sup\{\abs{y^*(x)} \, : \, y^*\in S_Y\}$ which is
smaller than or equal to the original norm, write $x^*=y^*+z^*$ with
$y^*\in Y$ and $z^*\in Z$, and observe that
$$
1=\re x^*(x)=\re y^*(x) + \re z^*(x) \leq \|y^*\|\|x\|_Y +
\|z^*\|\|x\| \leq \|y^*\|+\|z^*\|=1.
$$
Therefore, we have, in particular, that
\begin{equation}\label{eq:dual-ell1-sum}
\re y^*(x)=\|y^*\|\|x\|_Y.
\end{equation}
Now, we have
\begin{align*}
\left|\bigl(1-\sqrt{2\delta}\bigr) - \|x\|_Y \right|  =
\left|\bigl(1-\sqrt{2\delta}\bigr)\|y_0\|_Y - \|x\|_Y \right| \leq
\left\|\bigl(1-\sqrt{2\delta}\bigr) y_0 \,-\, x\right\|_Y
<\sqrt{2\delta}
\end{align*}
from which follows that $\|x\|_Y<1$ and so, $y^*=0$ by
\eqref{eq:dual-ell1-sum} and $\|z^*\|=\|x^*\|=1$. But then,
\begin{align*}
\|x_0^*-x^*\| = \left\|\frac{\sqrt{2\delta}}{2}\,y^*_0\right\| +
\left\|\Bigl(1-\frac{\sqrt{2\delta}}{2}\Bigr)z^*_0\, -\,
z^*\right\| \geq \frac{\sqrt{2\delta}}{2} +
\left|\Bigl(1-\frac{\sqrt{2\delta}}{2}\Bigr) \,-\, \|z^*\|\right|=
\sqrt{2\delta},
\end{align*}
a contradiction. Again, we have proved that $\Phi_X(\delta)\geq
\sqrt{2\delta}$, being the other inequality always true.
\end{proof}

Of course, the first consequence of the above result is to Banach
spaces which decompose as $\ell_\infty$-sum of two subspaces.
Indeed, if $X=Y\oplus_\infty Z$ for two (non-trivial) subspaces
$Y$ and $Z$, then $X^*=Y^\perp \oplus_1 Z^\perp$ and $Y^\perp$ and
$Z^\perp$ are $w^*$-closed, so far away of being dense. Therefore,
Proposition~\ref{prop-dual-ell-1-sum} applies. We have proved the
following result.

\begin{corollary}
Let $X$ be a Banach space. Suppose that there are two
(non-trivial) subspaces $Y$ and $Z$ such that $X=Y\! \oplus_\infty
\!Z$. Then $\Phi_X(\delta)=\Phi_X^S(\delta)=\sqrt{2\delta}$ for every $\delta\in(0,1/2]$.
\end{corollary}

As a consequence, we obtain the following examples, analogous to
the ones presented in Example~\ref{example:L_1(mu)}.

\begin{examples}
\begin{enumerate}
\item[(a)] Let $(\Omega,\Sigma,\mu)$ a measure space such that
$L_\infty(\Omega)$ has dimension greater than one and let $E$ be
any non-zero Banach space. Then,
    $$
    \Phi_{L_\infty(\mu,E)}=\Phi_{L_\infty(\mu,E)}^S(\delta)=\sqrt{2\delta} \qquad \bigl( \delta\in (0,1/2]\bigr).
    $$
\item[(b)] Let $\Gamma$ be a set with more than one point and let
$E$ be any non-zero Banach space. Then,
    $$
    \Phi_{c_0(\Gamma,E)}=\Phi_{c_0(\Gamma,E)}^S(\delta)=\sqrt{2\delta} \ \ \text{ and } \ \ \Phi_{c(\Gamma,E)}=\Phi_{c(\Gamma,E)}^S(\delta)=\sqrt{2\delta} \qquad \bigl( \delta\in (0,1/2]\bigr).
    $$
\end{enumerate}
\end{examples}

Our next aim is to deduce from
Proposition~\ref{prop-dual-ell-1-sum} that also arbitrary $C(K)$
spaces have the maximum moduli and for this we have to deal with
the concept of $M$-ideal. Given a subspace $J$ of a Banach space
$X$, $J$ is called \emph{$M$-ideal} if $J^\perp$ is a $L$-summand
on $X^*$ (use \cite{HWW} for background). In this case,
$X^*=J^\perp \oplus_1\!J^\sharp$ where $J^\sharp=\{x^*\in
X^*\,:\,\|x^*\|=\|x^*|_J\|\}\equiv J^*$. Now, if $X$ contain a
non-trivial $M$-ideal $J$, one has $X^*=J^\perp \oplus_1 J^\sharp$
and to apply Proposition~\ref{prop-dual-ell-1-sum} we need that
$J^\sharp$ to be not $\sigma(X^*,X)$-dense. Actually, $J^\sharp$
is not dense in $X^*$ if and only if there is $x_0\in
X\setminus\{0\}$ such that $\|x_0+y\|=\max\{\|x_0\|,\|y\|\}$ for
every $y\in J$ (this is easy to verify and a proof can be found in
\cite{Behrends}). Let us enunciate what we have shown.

\begin{corollary}\label{cor-M-ideal-Mcomplement}
Let $X$ be a Banach space. Suppose that there is a non-trivial
$M$-ideal $J$ of $X$ and a point $x_0\in X\setminus \{0\}$ such that
$\|x_0+y\|=\max\{\|x_0\|,\|y\|\}$ for every $y\in J$. Then,
$\Phi_X(\delta)=\Phi_X^S(\delta)=\sqrt{2\delta}$ for every $\delta\in (0,1/2]$.
\end{corollary}

With the above corollary we are able to prove that the moduli of
any non-trivial $C_0(L)$ space are maximum.

\begin{example}
{\slshape Let $L$ be a locally compact Hausdorff topological space
with at least two points and let $E$ be any non-zero Banach space. Then
$\Phi_{C_0(L,E)}(\delta)=\Phi_{C_0(L,E)}^S(\delta)=\sqrt{2\delta}$ for every $\delta\in (0,1/2]$.}
\newline Indeed, we may find a non-empty non-dense open subset $U$
of $L$ and consider the subspace
$$
J=\bigl\{f\in
C_0(L,E)\,:\,f|_{U}=0\},
$$
which is an $M$-ideal of $C_0(L,E)$ by \cite[Corollary~VI.3.4]{HWW} (use the simpler \cite[Example~I.1.4.a]{HWW} for the scalar-valued case) and it is non-zero since $L\setminus U$ has non-empty interior. As $U$ is open and non-empty, we may find a non-null function $x_0\in C_0(L,E)$ whose support is contained in $U$. It follows that $\|x_0+y\|=\max\{\|x_0\|,\|y\|\}$ for every $y\in J$ by disjointness of the supports.
\end{example}

A sufficient condition to be in the hypotheses of
Corollary~\ref{cor-M-ideal-Mcomplement} is that a Banach space $X$
contains two non-trivial $M$-ideals $J_1$ and $J_2$ such that
$J_1\cap J_2=\{0\}$ since, in this case, $J_1$ and $J_2$ are
complementary $M$-summands in $J_1+J_2$
\cite[Proposition~I.1.17]{HWW}. Let us comment that this is actually
what happens in $C(K)$ when $K$ has more than one point.

\begin{corollary}\label{cor:twoMideals}
Let $X$ be a Banach space. Suppose there are two non-trivial
$M$-ideals $J_1$ and $J_2$ such that $J_1\cap J_2=\{0\}$. Then
$\Phi_X(\delta)=\Phi_{X}^S(\delta)=\sqrt{2\delta}$ for every $\delta\in (0,1/2]$.
\end{corollary}

A sufficient condition for a Banach space to have two
non-intersecting $M$-ideals is that its centralizer is non-trivial
(i.e.\ has dimension at least two). We are not going into details,
but roughly speaking, the \emph{centralizer} $Z(X)$ of a Banach
space $X$ is a closed subalgebra of $L(X)$ isometrically
isomorphic to $C(K_X)$ where $K_X$ is a Hausdorff topological
space, and it is possible to see $X$ as a $C(K_X)$-submodule of
$\prod_{k\in K_X} X_k$ for suitable $X_k$'s. We refer to \cite[\S
3.B]{Behrends-book} and \cite[\S I.3]{HWW} for details. It happens
that every $M$-ideal of $C(K_X)$ produces an $M$-ideal of $X$ in a
suitable way (see \cite[\S 4.A]{Behrends-book}) and if $Z(X)$
contains more than one point, then two non-intersecting $M$-ideals
appear in $X$, so our corollary above applies.

\begin{corollary}
Let $X$ a Banach space. If $Z(X)$ has dimension greater than one,
then $\Phi_X(\delta)=\Phi_{X}^S(\delta)=\sqrt{2\delta}$ for every $\delta\in (0,2]$.
\end{corollary}

To give some new examples coming from this corollary, we recall
that the centralizer of a unital (complex) $C^*$-algebra
identifies with its center (see \cite[Theorem~V.4.7]{HWW} or
\cite[Example 3 in page 63]{Behrends-book}).

\begin{example}
{\slshape Let $A$ be a unital $C^*$-algebra with non-trivial
center. Then, $\Phi_A(\delta)=\Phi_{A}^S(\delta)=\sqrt{2\delta}$ for every $\delta\in
(0,1/2]$.}
\end{example}

It would be interesting to see whether the algebra $L(H)$ for a
finite- or infinite-dimensional Hilbert space $H$ has the maximum
Bishop-Phelps-Bollob\'{a}s moduli. None of the results of this
section applies to it since its center is trivial and, despite it
contains $K(H)$ as an $M$-ideal, there is no element $x_0\in L(H)$
satisfying the requirements of
Corollary~\ref{cor-M-ideal-Mcomplement} (see
\cite[page~538]{Behrends}). Let us also comment that the bidual of
$L(H)$ is a $C^*$-algebra with non-trivial centralizer, so
$\Phi_{L(H)^{**}}(\delta)=\Phi_{L(H)^{**}}^S(\delta)=\sqrt{2\delta}$ for every $\delta \in
(0,1/2]$. If there is $\delta\in (0,1/2]$ such that
$\Phi_{L(H)}(\delta)<\sqrt{2\delta}$, then this would be an
example when the inequality in
Proposition~\ref{proposicion-modulo-dual} is strict.

We finish this section with two pictures: one with the
Bishop-Phelps-Bollob\'{a}s moduli of $\R$, $\C$ and $\ell_\infty^{(2)}$, and another one with the corresponding values of the spherical Bishop-Phelps-Bollob\'{a}s moduli.

\begin{figure}[htbl!]
\definecolor{ffqqqq}{rgb}{1,0,0}
\definecolor{qqffqq}{rgb}{0,1,0}
\definecolor{qqqqff}{rgb}{0,0,1}
\definecolor{ffqqff}{rgb}{1,.5,0}
\begin{minipage}[t]{0.45\linewidth}
\centering
\begin{tikzpicture}[line cap=round,line join=round,>=triangle 45,x=3.0cm,y=3.0cm]
\draw[->,color=black] (0,0) -- (2.1,0);
\foreach \x in {1,2}
\draw[shift={(\x,0)},color=black] (0pt,2pt) -- (0pt,-2pt) node[below] {\footnotesize $\x$};
\draw[->,color=black] (0,0) -- (0,2.1);
\foreach \y in {,1,2}
\draw[shift={(0,\y)},color=black] (2pt,0pt) -- (-2pt,0pt) node[left] {\footnotesize $\y$};
\draw[color=black] (0pt,-10pt) node[right] {\footnotesize $0$};
\draw[line width=1.6pt,color=qqqqff, smooth,samples=100,domain=0.0:2.0] plot(\x,{sqrt(2*(\x))});
\draw[line width=1.6pt,color=qqffqq, smooth,samples=100,domain=1.0:2.0] plot(\x,{sqrt((\x)-1)+1});
\draw[line width=1.6pt,color=ffqqqq, smooth,samples=100,domain=0.0:0.539189] plot(\x,{sqrt(2-sqrt(4-2*(\x)))});
\draw[line width=1.6pt,color=ffqqqq, smooth,samples=100,domain=1.0:2.0] plot(\x,{sqrt((\x))});
\draw[line width=1.6pt,color=qqffqq, smooth,samples=100,domain=0.0:0.539189] plot(\x,{(\x)});
\draw[line width=1.6pt,color=ffqqqq, smooth,samples=100,domain=0.539189:1.0] plot(\x,{(\x)});
\draw[line width=1.6pt,color=ffqqff,smooth,samples=100,domain=0.53919:1.0] plot(\x,{(\x)});
\label{fig:modulii}
\end{tikzpicture}
\caption{The value of $\Phi_X(\delta)$ for \textcolor{green}{$\boldsymbol{\R}$}, \textcolor{red}{$\boldsymbol{\C}$} and \textcolor{blue}{$\boldsymbol{\ell_\infty^{(2)}}$}}
\end{minipage}
\hfill
\begin{minipage}[t]{0.45\linewidth}
\centering
\begin{tikzpicture}[line cap=round,line join=round,>=triangle 45,x=3.0cm,y=3.0cm]
\draw[->,color=black] (0,0) -- (2.1,0);
\foreach \x in {1,2}
\draw[shift={(\x,0)},color=black] (0pt,2pt) -- (0pt,-2pt) node[below] {\footnotesize $\x$};
\draw[->,color=black] (0,0) -- (0,2.1);
\foreach \y in {,1,2}
\draw[shift={(0,\y)},color=black] (2pt,0pt) -- (-2pt,0pt) node[left] {\footnotesize $\y$};
\draw[color=black] (0pt,-10pt) node[right] {\footnotesize $0$};
\draw[line width=1.6pt,color=qqqqff, smooth,samples=100,domain=0.0:2.0] plot(\x,{sqrt(2*(\x))});
\draw[line width=1.6pt,color=qqffqq, smooth,samples=100,domain=0.0:2.0] plot(\x,{0});
\draw[line width=1.6pt,color=ffqqqq, smooth,samples=100,domain=0.0:2.0] plot(\x,{sqrt(2-sqrt(4-2*(\x)))});
\label{fig:modulii-sphere}
\end{tikzpicture}
\caption{The value of $\Phi^S_X(\delta)$ for \textcolor{green}{$\boldsymbol{\R}$}, \textcolor{red}{$\boldsymbol{\C}$} and \textcolor{blue}{$\boldsymbol{\ell_\infty^{(2)}}$}}
\end{minipage}
\end{figure}

\section{Banach spaces with the greatest possible modulus}\label{sec:5}
Our goal in this section is to show that Banach spaces with the greatest possible moduli contain almost isometric copies of the real $\ell_\infty^2$. Let us first recall the following definition.

\begin{definition} \label{def-alm-isom}
Let $X$, $E$ be Banach spaces. $X$ is said to contain almost isometric copies of $E$ if, for every $\eps > 0$ there is a subspace $E_\eps \subset X$ and there is a bijective linear operator $T: E \longrightarrow E_\eps$ with $\|T\| < 1 + \eps$ and $\|T^{-1}\| < 1 + \eps$.
\end{definition}

The next result is well-known and has a straightforward proof.

\begin{lemma}\label{lema-ell-infty-cop}
A real Banach space $E$ contains an isometric copy of
$\ell_\infty^{(2)}$ if and only if there are elements $u, v \in S_E$ such that $\|u - v\| = \|u + v\| = 2$. $E$ contains almost
isometric copies of $\ell_\infty^{(2)}$ if and only if there are
elements $u_n, v_n \in S_E$, $n\in \N$ such that $\|u_n -
v_n\| \longrightarrow 2$ and $\|u_n + v_n\| \longrightarrow 2$ as $n \to \infty$.
\end{lemma}

The class of spaces $X$ that do not contain almost isometric copies of $\ell_\infty^{(2)}$ was deeply studied by James \cite{James1964} (see also the exposition in Van Dulst's book \cite{Van Dulst1978}), who gave to such spaces the name ``uniformly non-square''.  He proved in particular, that every uniformly non-square space must be reflexive, that this property is stable under passing to subspaces, quotient spaces and duals. In fact, a general result is true \cite{Kadets1982}: for every 2-dimensional space $E$ if a real Banach space $X$ does not contain almost isometric copies of $E$ then $X$ is reflexive.

The aim of this section is to prove that if a real Banach space
$X$ satisfies that its Bishop-Phelps-Bollob\'as modulus is
$\sqrt{2\delta}$ in at least one point $\delta \in (0, 1/2)$, then
$X$ (and, equivalently, the dual space) contains almost isometric
copies of $\ell_\infty^{(2)}$. Actually, as shown in Remark~\ref{prop-Phi-Phi^s}, $\Phi_X(\delta)=\sqrt{2\delta}$ if and only if $\Phi_X^S(\delta)=\sqrt{2\delta}$. Therefore, we may use the formally stronger hypothesis of $\Phi_X^S(\delta)=\sqrt{2\delta}$.

We will use some lemmas and ideas of Bishop and Phelps
\cite{Bishop-Phelps}, but for the reader's convenience we will
refer to the corresponding lemmas in the already classical Diestel's book \cite{Diestel}.

From now on, $X$ will denote a \emph{real} Banach space. For $t > 1$ and $x^*\in S_{X^*}$, we denote
$$
K(t,x^*):=\{x\in X \, : \,
\norm{x}\leq t\, x^*(x)\}.
$$
Observe that $K(t,x^*)$ is a convex cone with non-empty interior.

\begin{lemma}[\mbox{A particular case of \cite[Chapter 1, Lemma
1]{Diestel}}]\label{lema-Diest-1} For every $z \in B_X$, every $x^*\in S_{X^*}$ and
every $t > 1$, there is $x_0 \in S_X$ such that $x_0 - z \in
K(t,x^*)$ and $[K(t,x^*) + x_0] \cap B_X = \{x_0\}$.
\end{lemma}

\begin{lemma}[\mbox{\cite[Chapter 1, Lemma
2]{Diestel} with a little modification that follows from the proof
there}]\label{lema-Diest-2} Let $x^*, y^* \in S_{X^*}$ and suppose that $x^*\bigl(\ker y^*
\cap S_X\bigr) \subset (- \infty, \eps/2]$. Then
$$
\dist\bigl(x^*, \lin y^*\bigr) \leq \eps/2 \qquad \text{and} \qquad \min\{\|x^* - y^*\|, \|x^* + y^*\|\} \leq
\eps.
$$
\end{lemma}

\begin{lemma}\label{lema-Diest-3}
Let $z \in B_X$, $x^*\in S_{X^*}$, $t > 1$, and let $x_0 \in S_X$
be from Lemma \ref{lema-Diest-1}. Denote $y^* \in S_{X^*}$ a
functional that separates $x_0 + K(t,x^*)$ from $B_X$, so $y^*(x_0) = 1$ and $y^*\bigl(K(t,x^*)\bigr) \subset [0 , \infty)$. Then $x^*\bigl(\ker y^* \cap S_X\bigr) \subset (- \infty, 1/t]$ and so,  $\dist\bigl(x^*, \lin y^*\bigr) \leq
1/t$ and $\min\{\|x^* - y^*\|, \|x^* + y^*\|\} \leq 2/t$.
\end{lemma}

\begin{proof}
This also can be extracted from \cite[Chapter 1]{Diestel}, but it
is better to give a proof. For every $w \in \ker y^* \cap S_X$ we
have that $w$ does not belong to the interior of $K(t,x^*)$, so $1
= \norm{w}\geq t \, x^*(w)$, i.e.\ $x^*(\ker y^* \cap S_X) \subset
(- \infty, 1/t]$. An application of Lemma~\ref{lema-Diest-2}
completes the proof.
\end{proof}

Now we are passing to our results. At first, for the sake of
simplicity, we consider the easier finite-dimensional case.

\begin{lemma}\label{lema-previo-1-simple}
Let $X$ be a finite-dimensional real space. Fix $\varepsilon \in (0, 1)$. Suppose that $(x,x^*)\in S_X\times S_{X^*}$ satisfies that
$x^*(x)=1-\frac{\varepsilon^2}{2}$ and that
$$
\max\{\norm{y-x},\norm{y^*-x^*}\}\geq \varepsilon
$$
for every pair $(y,y^*)\in \Pi(X)$. Then for $t
=\frac{2}{\varepsilon}$, there exists $y_0 \in \left[x+
K(t,x^*)\right] \cap S_X$ such that $x^*(y_0) = 1$.
\end{lemma}

\begin{proof}
Consider a sequence $t_n > t$, $n\in \N$, with $\lim_n t_n = t$. Using Lemma~\ref{lema-Diest-1}, we get $y_n\in S_X$ such that
\begin{equation}\label{eq:yn-1-infinita}
y_n-x\in K(t_n,x^*)\quad \text{ and } \quad \left( K(t_n,x^*)+ y_n\right)\cap B_X = \{y_n\}.
\end{equation}
Let $y_n^*\in X^*$ be a functional that separates $K(t_n,x^*)+
y_n$ from $B_X$, i.e. $y_n^*(y_n) = 1$ and $y_n^*(K(t_n,x^*))
\subset [0 , \infty)$. Then, according to Lemma~\ref{lema-Diest-3},
\begin{equation}\label{eq:two-cases}
\min\{\|x^* - y_n^*\|, \|x^* + y_n^*\|\} \leq 2/t_n < \eps.
\end{equation}
But
\begin{equation*}
\|x^* + y_n^*\| \geq (x^* + y_n^*)(y_n) = 1 + x^*(y_n) = 1 +
x^*(x)+ x^*(y_n - x) = 2-\frac{\varepsilon^2}{2} + x^*(y_n - x).
\end{equation*}
Since $(y_n - x) \in K(t_n,x^*)$, we have $x^*(y_n - x) \geq
\norm{(y_n - x)}/t_n \geq 0$ so
$$
\|x^* + y_n^*\| \geq 2-\frac{\varepsilon^2}{2} > \eps
$$
(we have used here that $0<\eps<1$). Comparing with \eqref{eq:two-cases}, we get $\|x^* - y_n^*\| <
\eps$, so the condition of our lemma says that $\|x - y_n\| \geq
\eps$. Without loss of generality (passing to a subsequence if
necessary) we can assume that $y_n$ tend to some $y_0$. Then
\begin{align*} \label{align(i)}
\eps &\leq \lim_n \|y_n - x\| \leq \lim_n t_n x^*(y_n - x) =
t(x^*(y_0) - x^*(x)) \\
&\leq \frac{2}{\eps}(x^*(y_0) - 1 + \frac{\eps^2}{2}) \leq
\frac{2}{\eps}(1 - 1 + \frac{\eps^2}{2}) = \eps.
\end{align*}
This means that all the inequalities in the above chain are in
fact equalities. In particular, $x^*(y_0) = 1$ and
$$
\|y_0 - x\| = \lim_n \|y_n - x\| = t\bigl(x^*(y_0) - x^*(x)\bigr),
$$
i.e.\ $y_0 \in \left[x+ K(t,x^*)\right] \cap S_X$.
\end{proof}

\begin{lemma}\label{lema-previo-2-simple}
Under the conditions of Lemma~\ref{lema-previo-1-simple}, there are
$y^* \in S_{X^*}$ and $\alpha \geq 1 - \frac{\eps}{2}$ with
\begin{equation}\label{eq:eps/2-eps}
\|x^* - \alpha y^*\| \leq \frac{\eps}{2} \,\,\,\textrm{ and }
\,\,\, \|x^* - y^*\| \geq \eps,
\end{equation}
and there is $v \in S_{X}$ such that
\begin{equation}\label{eq:x*=y*=1}
x^*(v) = y^*(v) = 1.
\end{equation}
\end{lemma}

\begin{proof}
Let $y_0$ be from the previous lemma. Fix a strictly increasing
sequence of $t_n > 1$ with $\lim_n t_n=t$ and let us consider two
cases.

\noindent\emph{\emph{Case 1\,:}} Suppose there exists $m_0\in \N$ with
$\mbox{int}\bigl[K(t_{m_0}, x^*)+x \bigr]\cap B_X \neq
\emptyset$. Then, using the fact that for every closed convex set
with non-empty interior, the closure of the interior is the whole
set, we get
$$
y_0 \in \bigl[x+ K(t,x^*)\bigr] \cap B_X =
\overline{\mbox{int}\bigl[x+K(t,x^*)\bigr]\cap B_X}=
\overline{\bigcup_{n \geq
m_0}\mbox{int}\bigl[x+K(t_n,x^*)\bigr]\cap B_X}.
$$
So, we can pick
\begin{equation}\label{eq:noname0}
z_n \in \bigl[x+K(t_n,x^*)\bigr]\cap B_X
\end{equation}
such that $z_n \longrightarrow y_0$.
In particular, $x^*(z_n) \longrightarrow 1$. Let us apply Lemma~\ref{lema-Diest-1}: there are $v_n \in S_X$ such that
\begin{equation}\label{eq:noname1}
v_n - z_n \in K(t_n,x^*)\quad \text{ and } \quad \bigl[K(t_n,x^*) + v_n\bigr] \cap B_X =
\{v_n\}.
\end{equation}
Then $x^*(v_n - z_n) \geq 0$, i.e.\ $1 \geq x^*(v_n) \geq x^*(z_n) \longrightarrow 1$, so $x^*(v_n) \longrightarrow 1$. Condition \eqref{eq:noname0} implies that $z_n - x \in K(t_n,x^*)$ which, together with \eqref{eq:noname1}, mean that $v_n - x \in K(t_n,x^*)$. Consequently,
$$
\|v_n - x\| \leq t_n x^*(v_n - x) \leq t_n \frac{\eps^2}{2} < \eps.
$$
If we denote $y_n^* \in S_{X^*}$ to the functional that separates
$v_n + K(t_n,x^*)$ from $B_X$, then $(v_n, y_n^*) \in \Pi(X)$. Since we are working under the conditions of Lemma~\ref{lema-previo-1-simple}, it follows that
$$
\|y_n^* - x^*\| \geq \eps.
$$
Also, by Lemma \ref{lema-Diest-3}, $\dist(x^*, \lin y_n^*) \leq
1/t_n$, so there are $\alpha_n \in \R$ such that
$$
\|x^* - \alpha_n y_n^*\|\leq 1/t_n.
$$
Again, without loss of generality, we may assume that the sequences
$(\alpha_n)$, $(v_n)$ and $(y_n^*)$ have limits. Let us denote
$\alpha:=\lim_n \alpha_n$, $y^*:=\lim_n y_n^*$, and $v:=\lim_n
v_n$. Then $\|v\| = 1$, $\|y^*\| = 1$, $x^*(v) = \lim_n x^*(v_n) =
1$, and $y^*(v) = \lim_n y_n^*(v_n) = 1$. This proves
\eqref{eq:x*=y*=1}. Also,
$$
\|x^* - \alpha y^*\| = \lim_n \|x^* - \alpha_n y_n^*\| \leq
\frac{1}{t} = \frac{\eps}{2}.
$$
Consequently,
\begin{equation}\label{eq:noname2}
\frac{\eps}{2} \geq \|x^* - \alpha y^*\| \geq (x^* - \alpha y^*)(v)
= 1 - \alpha,
\end{equation}
so, $\alpha \geq 1 - \frac{\eps}{2}$.

\noindent\emph{Case 2\,:} Assume that for every $n \in \N$
we have $\mbox{int}\bigl[K(t_{n}, x^*)+x \bigr]\cap B_X =
\emptyset$. Let us separate $x + \mbox{int}\left(K(t_{n}, x^*)
\right)$ from $B_X$ by a norm-one functional $y_n^*$, that is,
$$
y_n^*\left(x + \mbox{int}\bigl[K(t_{n}, x^*) \bigr]\right) > 1,
$$
so, in particular, $y_n^*(x) \geq 1$.

Again, passing to a subsequence, we can assume that there exists
$y^* = \lim_n y_n^*$ which satisfies $\|y^*\| = 1$, $1 \geq y^*(x) \geq \lim_n y_n^*(x) \geq 1$. So, $y^*(x) = 1$, i.e.\ $(x,y^*)\in \Pi(X)$. By the conditions of our lemma, this implies that
$$
\norm{y^*-x^*} = \max\{\norm{x-x},\norm{y^*-x^*}\}\geq
\varepsilon.
$$
Since
$$
y_0 \in x+ K(t,x^*)  = \overline{\bigcup_{n \in
\N}\mbox{int}\bigl[x+K(t_n,x^*)\bigr]},
$$
we can select $z_n \in \mbox{int}\bigl[x+K(t_n,x^*)\bigr]$ in such a
way that $z_n \longrightarrow y_0$. Then
$$
y^*(y_0) = \lim_n y_n^*(z_n) \geq 1,
$$
hence, $y^*(y_0) = 1$. This means that condition
\eqref{eq:x*=y*=1} works for $v := y_0$. The remaining conditions
can be deduced from Lemma \ref{lema-Diest-3} the same way as in the case 1.
\end{proof}

We are now able to state and prove the main result of the section in the finite-dimensional case.

\begin{theorem} \label{thm-finite-dim}
Let $X$ be a finite-dimensional real Banach space. Suppose that
there is a $\delta \in(0, 1/2)$ such that $\Phi_X(\delta)=\sqrt{2
\delta}$ (or, equivalently, $\Phi_X^S(\delta)=\sqrt{2
\delta}$). Then $X^*$ contains an isometric copy of
$\ell_\infty^{(2)}$ (hence, $X$ also contains an isometric copy of
$\ell_\infty^{(2)}$).
\end{theorem}

\begin{proof}
Denote $\eps: = \sqrt{2 \delta}\in (0,1)$. There is a sequence of pairs $(x_n, x_n^*)
\in S_X \times S_{X^*}$ such that $x_n^*(x_n) > 1 - \delta = 1 -
\frac{\eps^2}{2}$ and
$$
\max\{\norm{y-x_n},\norm{y^*-x_n^*}\} \geq \varepsilon -  \frac{1}{n}
$$
for every pair $(y,y^*)\in \Pi(X)$. Since the space is
finite-dimensional, we can find a subsequence of $(x_n, x_n^*)$
that converges to a pair $(x, x^*)\in S_X \times S_{X^*}$. This
pair satisfies that $x^*(x) \geq 1 - \delta$ and for every $(y,y^*)\in \Pi(X)$,
\begin{align*}
\max\{\norm{y-x},\norm{y^*-x^*}\} &\geq
\max\{\norm{y-x_n},\norm{y^*-x_n^*}\} -
\max\{\norm{x-x_n},\norm{x^*-x_n^*}\} \\ & \geq \varepsilon -  \frac{1}{n} -
\max\{\norm{x-x_n},\norm{x^*-x_n^*}\} \longrightarrow \eps.
\end{align*}
Since by Theorem~\ref{Theorem B-P-B
improved}, $x^*(x)$ cannot be strictly smaller than $1
- \delta$, we have $x^*(x) = 1 - \delta$. Therefore, we may apply Lemma~\ref{lema-previo-2-simple} to find $y^* \in S_{X^*}$ and $\alpha \geq 1 - \frac{\eps}{2}$ for which conditions \eqref{eq:eps/2-eps} and \eqref{eq:x*=y*=1} are fulfilled. Now we \emph{claim} that
in fact there is only one number $\gamma \in \R$ for which
\begin{equation}\label{eq:noname4}
\|x^* - \gamma y^*\| \leq \frac{\eps}{2}
\end{equation}
and this $\gamma$ equals $1 - \frac{\eps}{2}$. So $\alpha = 1 -
\frac{\eps}{2}$ and, we also \emph{claim} that
\begin{equation}\label{eq:noname5}
\|x^* - \alpha y^*\| = \frac{\eps}{2} \quad \text{ and } \quad
\|x^* - y^*\| = \eps.
\end{equation}
Indeed, when we were proving equation \eqref{eq:noname2}, we
proved that every $\gamma \in \R$ that satisfies
\eqref{eq:noname4} must satisfy $\gamma \geq 1 - \frac{\eps}{2}$.
On the other hand, the function
$\gamma \longmapsto \|x^* - \gamma y^*\|$ is convex, so the set $G$ of those $\gamma \in \R$ satisfying \eqref{eq:noname4} also must be
convex; but $1 \notin G$, so $\gamma < 1$. Finally, according to
\eqref{eq:eps/2-eps},
\begin{equation*}
\frac{\eps}{2} \geq 1 - \gamma = \|y^* - \gamma y^*\| \geq \|x^* -
y^*\| - \|x^* - \gamma y^*\| \geq \frac{\eps}{2}.
\end{equation*}
This means that all the inequalities above must be
equalities, so $\gamma \leq 1 - \frac{\eps}{2}$, and also \eqref{eq:noname5} must be true. The claim is proved.

Now, let us define
$$
u^* := \frac{x^* - \alpha y^*}{\|x^* - \alpha y^*\|} =
\frac{2}{\eps}(x^* - (1 - \frac{\eps}{2}) y^*),
$$
and let us show that functionals $u^*$ and $y^*$ span a subspace
of $X^*$ isometric to $\ell_\infty^{(2)}$. According to Lemma~\ref{lema-ell-infty-cop}, it is sufficient to show that
$\|u^* - y^*\| = \|u^* + y^*\| = 2$. Let us do this. At first,
$$
\|u^* - y^*\| = \left\|\frac{2}{\eps}(x^* - (1 - \frac{\eps}{2}) y^*) - y^*\right\| = \frac{2}{\eps}\|x^* -  y^*\| = 2.
$$
At second,
\begin{equation*}
2 \geq \|u^* + y^*\| = \|\frac{2}{\eps}(x^* - (1 - \frac{\eps}{2})
y^*) + y^*\| = \frac{2}{\eps}\|x^* -  y^* + \eps y^*\| \geq
\frac{2}{\eps}(x^* -  y^* + \eps y^*)(v) = 2.\qedhere
\end{equation*}
\end{proof}

Let us comment that for complex Banach spaces, we cannot expect that Theorem~\ref{thm-finite-dim} provides a \emph{complex} copy of $\ell_\infty^{(2)}$ in the dual of the space. Namely, the two-dimensional complex space $X=\ell_1^{(2)}$ satisfies $\Phi_X(\delta)=\sqrt{2\delta}$ for $\delta\in (0,1/2)$ but it does not contain the complex space $\ell_\infty^{(2)}$ (of course, it contains the real space $\ell_\infty^{(2)}$ as a subspace since $\ell_1^{(2)}$ and $\ell_\infty^{(2)}$ are isometric in the real case). We do not know whether it is true a result saying that if a complex space $X$ satisfies $\Phi_X(\delta)=\sqrt{2\delta}$ for some $\delta\in (0,1/2)$, then $X$ contains a copy of the complex space $\ell_1^{(2)}$ or a copy of the complex space $\ell_\infty^{(2)}$.

Let us extend the result of Theorem~\ref{thm-finite-dim} to the infinite-dimensional case. Roughly speaking, we proceed as in the proof of such theorem, but instead of selecting
convergent subsequences, we select subsequences such that their numerical characteristics (like norms of elements, pairwise distances, or values of some important functionals) have limits.

\begin{theorem} \label{thm-ellinfty2-general}
Let $X$ be an infinite-dimensional Banach space. Suppose that there is $\delta \in(0, 1/2)$ such that $\Phi_X(\delta)=\sqrt{2 \delta}$  (or, equivalently, $\Phi_X^S(\delta)=\sqrt{2
\delta}$). Then $X^*$ (and hence also $X$) contains almost isometric copies of $\ell_\infty^{(2)}$.
\end{theorem}

\begin{proof}
Denote $\eps: = \sqrt{2 \delta}$. There is a sequence of pairs $(x_n, x_n^*)
\in S_X \times S_{X^*}$ such that $x_n^*(x_n) > 1 - \delta = 1 -
\frac{\eps^2}{2}$ and
\begin{equation}\label{eq:like-lem1}
\max\{\norm{y-x_n},\norm{y^*-x_n^*}\} \geq \varepsilon - \frac{1}{n}
\end{equation}
for every pair $(y,y^*)\in \Pi(X)$.
Since we have $x_n^*(x_n) \leq 1 - (\varepsilon - \frac{1}{n})^2/2$ by Theorem~\ref{Theorem B-P-B improved}, we deduce that $\lim_n x_n^*(x_n) = 1 - \delta$.  Denote $t
=\frac{2}{\varepsilon}$. Now, we are going to proceed like in Lemma~\ref{lema-previo-1-simple} in order to show that there is a sequence $(y_n)$ of elements in $S_X$ such that
\begin{equation}\label{eq:yn-0-infinita+}
\lim_n\|y_n-x_n\| \leq t \lim_n x_n^*(y_n - x_n) \quad  \textrm{ and } \quad \lim_n x_n^*(y_n) = 1.
\end{equation}
Pick a sequence $(t_n)$ with $t_n > t$, $n\in \N$ and $\lim_n t_n = t$. Using Lemma~\ref{lema-Diest-1}, for every $n\in \N$ we get $y_n\in S_X$ such that
\begin{equation}\label{eq:yn-1-infinita+}
y_n-x_n\in K(t_n,x_n^*) \quad  \textrm{ and } \quad \left( K(t_n,x_n^*)+ y_n\right)\cap B_X = \{y_n\}.
\end{equation}
For given $n \in \N$, let $u_n^*\in S_{X^*}$ be a functional that separates $K(t_n,x_n^*)+
y_n$ from $B_X$, that is, satisfying $u_n^*(y_n) = 1$ and $u_n^*(K(t_n,x_n^*)) \subset [0 , \infty)$. Then, according to Lemma~\ref{lema-Diest-3}, we have
\begin{equation*}
\min\{\|x_n^* - u_n^*\|, \|x_n^* + u_n^*\|\} \leq 2/t_n < \eps.
\end{equation*}
As we have
\begin{equation*}
\|x_n^* + u_n^*\| \geq (x_n^* + u_n^*)(y_n) = 1 + x_n^*(y_n) = 1 +
x_n^*(x_n)+ x_n^*(y_n - x_n) \geq 2-\frac{\varepsilon^2}{2} > \eps,
\end{equation*}
we get $\|x_n^* - u_n^*\| < \eps$, so \eqref{eq:like-lem1} says that $\|x_n - y_n\| \geq \eps - \frac{1}{n}$. Without loss of generality, passing to a subsequence if necessary, we can assume that the following limits exist:
$\lim_n\|x_n - y_n\|$, $\lim_n x_n^*(y_n - x_n)$ and $\lim_n x_n^*(y_n)$. Then
\begin{align*} 
\eps &\leq \lim_n \|y_n - x_n\| \leq \lim_n t_n x_n^*(y_n - x_n) =
t \lim_n x_n^*(y_n - x_n) \\
&\leq \frac{2}{\eps}(\lim_n x_n^*(y_n) - 1 + \frac{\eps^2}{2}) \leq
\frac{2}{\eps}(1 - 1 + \frac{\eps^2}{2}) = \eps.
\end{align*}
This means that all the inequalities in the above chain are in
fact equalities. In particular, $\lim_n x_n^*(y_n) = 1$, and
\begin{equation}\label{eq:eps=lim}
\eps = \lim_n \|y_n - x_n\|  = t \lim_n x_n^*(y_n - x_n),
\end{equation}
so the analogue of Lemma~\ref{lema-previo-1-simple} is proved.

Now, we proceed with analogue of Lemma~\ref{lema-previo-2-simple}: we need to show that there are
$y_n^* \in S_{X^*}$ and $\alpha_n \geq 0$,  $\alpha_n \longrightarrow 1 -  \frac{\eps}{2}$ with
\begin{equation}\label{eq:eps/2-eps+}
\|x_n^* - \alpha_n y_n^*\| \leq \frac{\eps}{2} \,\,\,\textrm{ and }
\,\,\, \|x_n^* - y_n^*\| \geq \eps,
\end{equation}
and there is a sequence of $v_n \in S_{X}$ such that
\begin{equation}\label{eq:x*=y*=1+}
\lim_n x_n^*(v_n) = \lim_n y_n^*(v_n) = 1.
\end{equation}
\noindent\emph{Case 1\,:} Assume that there exist $r > 0$ and
$n \in \N$ such that, for all  $m > n$,
$$
\left(\bigl[K(t - r, x_{m}^*)+x_{m} \bigr]\cap B_X \right) \setminus \left(x_{m} + rB_X \right)\neq \emptyset.
$$
This means that for all $m > n$ there is $z_m$ such that
$$
\|z_m - x_m\| > r, \quad \|z_m\| \leq 1\quad \textrm{ and } \quad \|z_m - x_m\|  \leq (t-r) x_m^*(z_m - x_m).
$$
For $\lambda \in (0, 1)$ denote $y_{m, \lambda} := \lambda z_m + (1 - \lambda)y_m$. Clearly, $y_{m, \lambda} \in B_X$.
Denote also
$$
\lambda_m = \inf\{\lambda: y_{m, \lambda} \in x_m + K(t, x_m^*)\},
$$
and let us show that
\begin{equation}\label{eq:limlambdam}
\lim_m \lambda_m = 0.
\end{equation}
Observe first that $\lambda_m$ is smaller than every value of $\lambda$ for which
$$
\|y_{m, \lambda} - x_m\| \leq t x^*(y_{m, \lambda} - x_m).
$$
On the one hand, if $\|y_m - x_m\| - t x_m^*(y_m - x_m) \leq 0$, then $\lambda = 0$ belongs to the set in question, and the job is done. On the other hand, if $\|y_m - x_m\| - t x_m^*(y_m - x_m) \geq 0$, then there is $\lambda$ for which
$$
\lambda \|z_m - x_m\| + (1 - \lambda) \|y_m - x_m\| = t \lambda x_m^*(z_m - x_m) +t (1 - \lambda) x_m^*(y_m - x_m)
$$
is positive and belongs to the set in question. This means that
$$
\lambda_m \leq \frac{\|y_m - x_m\| - t x_m^*(y_m - x_m)}{\|y_m - x_m\| - t x_m^*(y_m - x_m) +t x_m^*(z_m - x_m)  - \|z_m - x_m\|},
$$
but the limit of the right-hand side equals $0$ thanks to \eqref{eq:eps=lim}. So condition \eqref{eq:limlambdam} is proved.
This means that
$y_{m, \lambda_m} \in x_m + K(t, x_m^*)$ and
$\|y_{m, \lambda_m} - y_m\| \leq 2  \lambda_m \longrightarrow 0$.
Let us pick a little bit bigger $\tilde{\lambda}_m > \lambda_m$ in such a way that we still have
$\|y_{m, \tilde{\lambda}_m} - y_m\|  \longrightarrow 0$, but for some $\tilde{t}_n < t$ with $\tilde{t}_n  \longrightarrow   t$, we have
\begin{equation}\label{eq:noname0+}
\|y_{m, \tilde{\lambda}_m} - x_m\| \leq \tilde{t}_nx_m^*(y_{m, \tilde{\lambda}_m} - x_m).
\end{equation}
Then, in particular, $\lim_n x_n^*(y_{n, \tilde{\lambda}_n}) = \lim_n x_n^*(y_n)=  1$. Let us apply Lemma~\ref{lema-Diest-1}. There are $v_n \in S_X$ such that
\begin{equation}\label{eq:noname1+}
v_n - y_{n, \tilde{\lambda}_n} \in K(\tilde{t}_n, x_n^*)\quad \textrm{ and }\quad \bigl[K(\tilde{t}_n, x_n^*) + v_n\bigr] \cap B_X = \{v_n\}.
\end{equation}
Then $x_n^*(v_n - y_{n, \tilde{\lambda}_n}) \geq 0$, i.e.\ $1 \geq x_n^*(v_n) \geq x_n^*(y_{n, \tilde{\lambda}_n}) \longrightarrow
1$, so $x_n^*(v_n) \longrightarrow  1$. This proves the first part of \eqref{eq:x*=y*=1+}. Condition \eqref{eq:noname0+} imply that
$y_{n, \tilde{\lambda}_n} - x_n \in K(\tilde{t}_n,x_n^*)$ which, together with \eqref{eq:noname1+}, mean that $v_n - x_n \in K(\tilde{t}_n, x_n^*)$. Consequently,
$$
\|v_n - x_n\| \leq \tilde{t}_n x_n^*(v_n - x_n) \leq \tilde{t}_n \frac{\eps^2}{2} < \eps.
$$
If we denote by $y_n^* \in S_{X^*}$ the functional that separates
$v_n + K(\tilde{t}_n,x^*)$ from $B_X$, then $(v_n, y_n^*) \in \Pi(X)$ (this proves the second part of
\eqref{eq:x*=y*=1+} even in a stronger form) so, thanks to \eqref{eq:like-lem1},
$$
\|y_n^* - x_n^*\| \geq \eps - \frac{1}{n}.
$$
Also, by Lemma~\ref{lema-Diest-3}, $\dist(x_n^*, \lin y_n^*) \leq
1/\tilde{t}_n$, so there are $\alpha_n \in \R$ such that
$$
\|x^* - \alpha_n y_n^*\|\leq 1/\tilde{t}_n.
$$
Again, without loss of generality, we may assume that the sequences
$(\alpha_n)$ and $\|x_n^* - \alpha_n y_n^*\| $ converge.
Then,
$$ \lim_n \|x_n^* - \alpha_n y_n^*\| \leq
\frac{1}{t} = \frac{\eps}{2}.
$$
Consequently,
\begin{equation*}
\frac{\eps}{2} \geq \lim_n \|x_n^* - \alpha_n y_n^*\|  \geq \lim_n (x_n^* - \alpha_n y_n^*)(v)
= 1 - \lim_n \alpha_n\,,
\end{equation*}
so $\lim_n \alpha_n \geq 1 - \frac{\eps}{2}$. Starting at this point, \eqref{eq:eps/2-eps+} can be deduced in the same way as it was done for \eqref{eq:noname5}.

\noindent\emph{Case 2\,:} Assume that there is a sequence of $r_n > 0$, $r_n \longrightarrow  0$ and that there is a subsequence of $(x_m, x_m^*)$ (that we will again denote  $(x_m, x_m^*)$)
such that
$$
\left(\bigl[K(t - r_m, x_{m}^*)+x_{m} \bigr]\cap B_X \right) \setminus \left(x_{m} + r_mB_X \right) =
\emptyset \qquad (\text{for all $m \in \N$}).
$$
Then also
$$
\bigl[K(t - r_m, x_{m}^*)+x_{m} \bigr]\cap (1 - r_m)B_X  =
\emptyset \qquad (\text{for all $m \in \N$}).
$$
Let us separate
$$
\frac{1}{ 1 - r_m}\bigl[K(t - r_m, x_{m}^*)+x_{m} \bigr]
$$
from $B_X$ by a norm-one functional $y_n^*$, that is,
\begin{equation}\label{eq-27-new}
y_n^*\bigl(K(t - r_m, x_{m}^*)+x_{m} \bigr) >  1 - r_m
\end{equation}
so, in particular, $y_m^*(x_m) \geq 1 -  r_m$ and
$\lim_m y_m^*(x_m) = 1$. By the Bishop-Phelps-Bollob\'{a}s theorem,
there is a sequence $(\tilde{x}_n, \tilde{y}_n^*)\in \Pi(X)$, such that
$$
\max\{\norm{\tilde{x}_n - x_n},\norm{\tilde{y}_n^* - y_n^*}\} \longrightarrow 0 \,\, \text{ as }\,\, n \to \infty.
$$
Again, passing to a subsequence, we can assume that all the numerical characteristics that appear here have the corresponding limits. According to \eqref{eq:like-lem1}, for $n$ big enough, we have
$$
\norm{\tilde{y}_n^* - x_n^*} = \max\{\norm{\tilde{x}_n - x_n},\norm{\tilde{y}_n^* - x_n^*}\} \geq \varepsilon - \frac1n,
$$
so $\lim_n\norm{y_n^*-x_n^*} \geq \varepsilon$.
We can select $z_n \in x_n + K(t - r_n,x_n^*)$ in such a
way that $\|z_n - y_n\| \longrightarrow 0$. Then
$$
1  \geq \lim_n y_n^*(y_n) = \lim_n y_n^*(z_n) \geq \lim_n(1-r_n) = 1.
$$
This means that condition
\eqref{eq:x*=y*=1+} works for $v_n := y_n$.

Now consider an arbitrary $w \in \ker y_n^* \cap S_X$. Taking a convex combination with an element $h$ of the unit sphere where $y_n^*(h)$ almost equals $-1$, we can construct an element $\tilde{w} \in B_X$ such that $\|\tilde{w} - w\| \leq 2 r_n$ and $y_n^*(\tilde{w}) = - r_n$.
Then, by (\ref{eq-27-new}), $\tilde{w} \notin \mbox{int}\bigl( K(t - r_n, x_{n}^*)\bigr)$, so $\|\tilde{w}\| \geq (t - r_n) x_{n}^*(\tilde{w})$.
Consequently,
$$
 x_{n}^*(w) \leq  x_{n}^*(\tilde{w}) + 2 r_n \leq \frac{1}{t - r_n} + 2 r_n.
$$
Observe that we have shown that the values of the functional $x_{n}^*$ on $\ker y_n^* \cap S_X$ do not exceed $\frac{1}{t - r_n} + 2 r_n$. Therefore, by Lemma~\ref{lema-Diest-2},
$$
\dist(x_n^*, \lin y_n^*) \leq \frac{1}{t - r_n} + 2 r_n \longrightarrow  \frac{1}{t}
$$
and so there are $\alpha_n \in \R$ such that
$$
\lim_n\|x_n^* - \alpha_n y_n^*\| \leq \frac{1}{t}.
$$
The remaining conditions in \eqref{eq:eps/2-eps+} and \eqref{eq:x*=y*=1+} can be deduced the same way as in the case 1.

Finally, \eqref{eq:eps/2-eps+} and \eqref{eq:x*=y*=1+} imply that $\lim_n\|x_n^* -  y_n^*\| = \lim_n\|x_n^* -  y_n^*\| = 2$: the proof does not differ much from the corresponding part of the Theorem~\ref{thm-finite-dim} demonstration.
\end{proof}


\begin{corollary} \label{cor-un-non-square}
Let $X$ be a uniformly non-square Banach space.  Then, $\Phi_X^S(\delta)\leq \Phi_X(\delta) < \sqrt{2 \delta}$ for every $\delta \in(0, 1/2)$. Consequently, every superreflexive Banach space can be equivalently renormed in such a way that, in the new norm, $\Phi_X^S(\delta)\leq\Phi_X(\delta) < \sqrt{2 \delta}$ for all $\delta \in(0, 1/2)$.
\end{corollary}

It would be interesting to obtain a quantitative version of the above corollary.

 \section{A three dimensional space $E$ containing $\ell_\infty^{(2)}$ with $\Phi_E(\delta) < \sqrt{2 \delta}$}
 \label{sec:6}

In the last section we decided to decorate our paper with two diamonds:

\begin{figure}[htbl!]
\begin{minipage}[t]{0.45\linewidth}
\centering
\includegraphics[width=5cm]{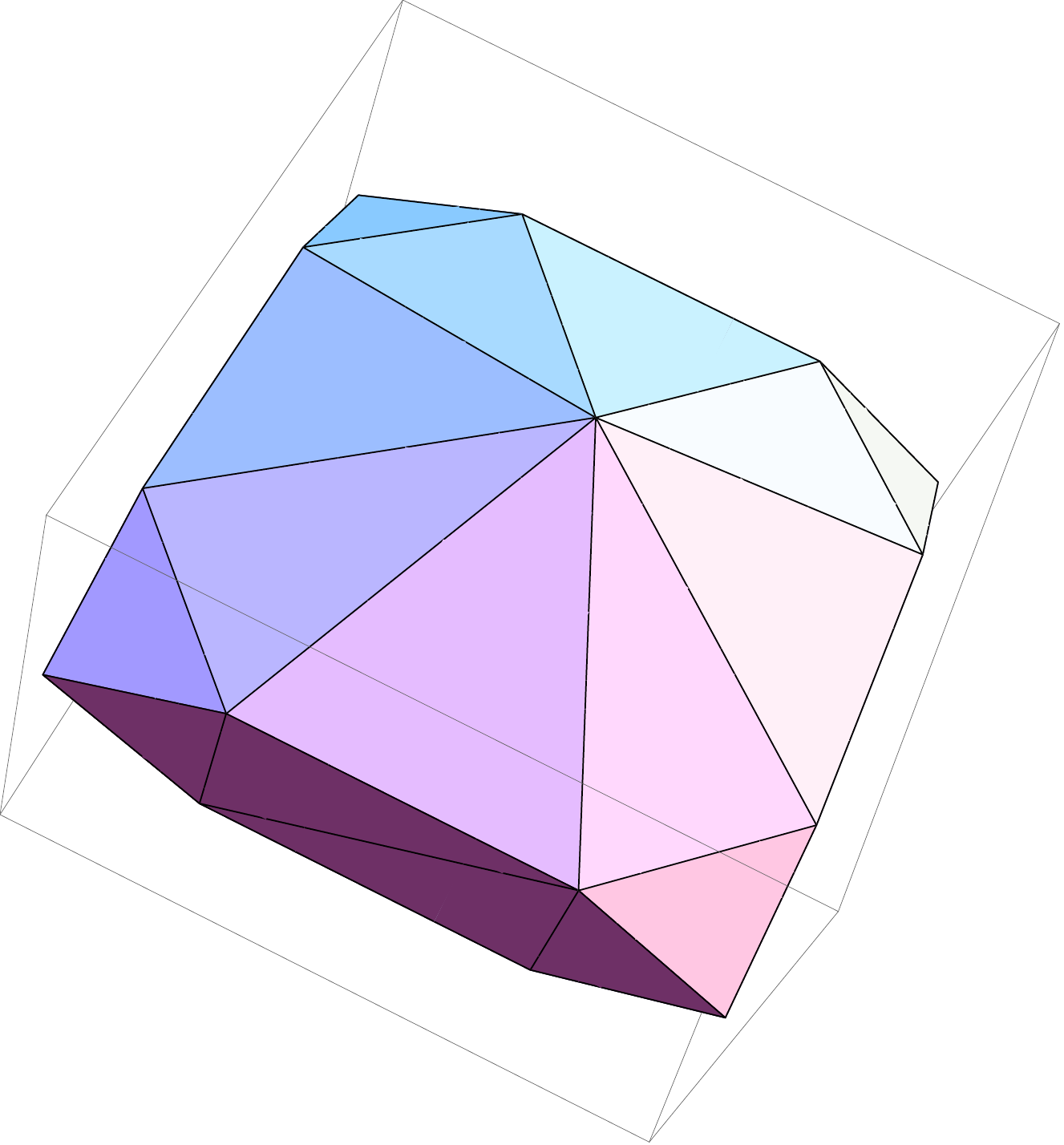}
\caption{The unit ball of $D_{\frac12}$}
\end{minipage}
\hfill
\begin{minipage}[t]{0.45\linewidth}
\centering
\includegraphics[width=5cm]{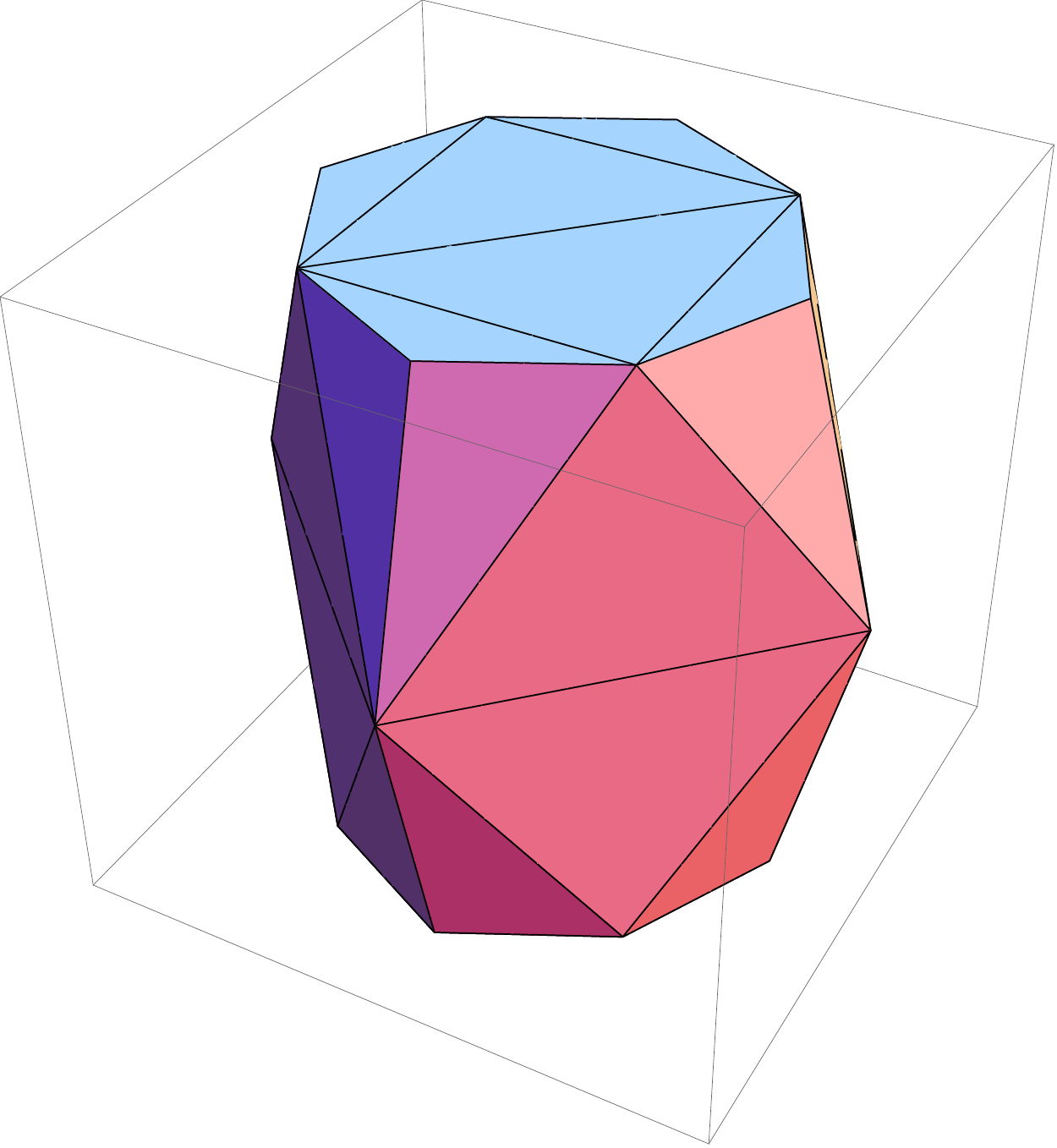}
\caption{The unit ball of $D_{\frac12}^*$}
\end{minipage}
\end{figure}

The first of them represents the unit ball of the space $D_\eps$ that we construct below, and on the second picture one can see the unit ball of  $D_\eps^*$.

Like in the previous section, for every $\delta \in (0, 1/2)$ we denote $\eps = \sqrt{2 \delta}$, so $0 < \eps < 1$. We denote $B_\eps^3 \subset \R^3$ the absolute convex hull of the following 11 points $A_k, k = 1, \ldots , 11$ (or, what is the same, the convex hull of 22 points $\pm A_k, k = 1, \ldots , 11$):
$$
A_1=(0,0,\frac34),
$$
$$
A_2=(1-\eps, 1, \frac\eps2), \  A_3=(1-\eps, - 1, \frac\eps2),  \    A_4=(\eps - 1,  1, \frac\eps2),  \  A_5=(\eps - 1, - 1, \frac\eps2),
$$
$$
 A_6=(1, 1-\eps,  \frac\eps2),  \  A_7=(-1, 1-\eps,  \frac\eps2),  \  A_8=(1, \eps - 1,  \frac\eps2),  \  A_9=(-1, \eps - 1,  \frac\eps2),
$$
$$
 A_{10}=(1, 1, 0), \  A_{11}=(1, - 1, 0).
$$
Denote $D_\eps$ (``D'' from ``Diamond'') the normed space $\left(\R^3, \|\cdot\|\right)$, for which $B_\eps^3$ is its unit ball. Then $D_\eps^*$ can be viewed as $\R^3$ with the polar of $B_\eps^3$ as the unit ball, and the action of $x^* \in D_\eps^*$ on $x \in D_\eps$ is just the standard inner product in $\R^3$.
Let us list, without proof, some properties of $D_\eps$ whose verification is straightforward:
\begin{itemize}
\item The subspace of $D_\eps$ formed by vectors of the form $(x_1, x_2, 0)$ is canonically isometric to $\ell_\infty^{(2)}$.
\item
There are no other isometric copies of $\ell_\infty^{(2)}$ in  $D_\eps$.
\item The subspace of $D_\eps^*$ formed by vectors of the form $(x_1, x_2, 0)$ is canonically isometric to $\ell_1^{(2)}$ (and so, is isometric to $\ell_\infty^{(2)}$).
\item
There are no other isometric copies of $\ell_\infty^{(2)}$ in  $D_\eps^*$.
\item
The following operators act as isometries both on  $D_\eps$ and  $D_\eps^*$: $(x_1, x_2, x_3) \longmapsto (x_2, x_1, x_3)$, $(x_1, x_2, x_3) \longmapsto (x_1, - x_2, x_3)$. In other words, changing the sign of one coordinate or rearranging the first two coordinates do not change the norm of an element.
\end{itemize}

The following theorem shows that the existence of an $\ell_\infty^{(2)}$-subspace does not imply that $\Phi_{X}(\delta)=\sqrt{2 \delta}$, even in dimension 3.

\begin{theorem} \label{thm-diamond-eps}
Let $\delta \in(0, 1/2)$, $\eps = \sqrt{2 \delta}$, and $X = D_\eps$. Then $\Phi_{X}(\delta) < \sqrt{2 \delta}$.
\end{theorem}

\begin{proof}
Assume contrary that
$\Phi_{X}(\delta)=\sqrt{2 \delta}$. Like in the proof of Theorem~\ref{thm-finite-dim}, this implies the existence of a pair $(x, x^*)\in S_X \times S_{X^*}$ with the following properties: $x^*(x) = 1 - \delta$ and
\begin{equation} \label{eq-diamond-contr}
\max\{\norm{z-x},\norm{z^*-x^*}\}  \geq  \eps \textrm{ for every pair } (z, z^*)\in \Pi(X).
\end{equation}
Also, repeating the proof of Theorem \ref{thm-finite-dim} for this $x^* \in S_{X^*}$, we can find $u^*, y^* \in S_{X^*}$ such that the pair $(u^*, y^*)$ is 1-equivalent to the canonical basis of $\ell_1^{(2)}$ and
$$
u^* =
\frac{2}{\eps}(x^* - (1 - \frac{\eps}{2}) y^*).
$$
This means that $x^* = \frac{\eps}{2}u^* + (1 - \frac{\eps}{2}) y^*$.
What can be this $(u^*, y^*)$ if we take into account that there is only one isometric copy of $\ell_1^{(2)}$ in $X^*$? It can be either $u^*=(1, 0, 0)$,
$y^* = (0, 1, 0)$, or a pair of vectors that can be obtained from this one by application of isometries, i.e.\ just 8 possibilities. Consequently, $x^*$ either equals to the vector $(\eps/2, 1 - \eps/2, 0)$, or to a vector that can be obtained from this one by application of isometries, again just 8 possibilities.

By duality argument, there are $u, y \in S_{X}$ such that the pair $(u, y)$ is 1-equivalent to the canonical basis of $\ell_1^{(2)}$ and
$$
x = \frac{\eps}{2}u + (1 - \frac{\eps}{2}) y.
$$
Since the only (up to isometries) pair $u, y \in S_{X}$ of this kind is $u=(1,1,0)$, $y=(1, -1,0)$, we get $x = (1, 1 - \eps, 0)$, or can be obtained from this one by application of isometries. So there are $8\times8 = 64$  possibilities for the pair $(x, x^*)$.  Taking into account that  $x^*(x) = 1 - \delta$ we reduce this number to 8 possibilities: $x = (1 - \eps, 1,  0)$, $x^*=(\eps/2, 1 - \eps/2, 0)$ and images of this pair under remaining 7 reflections and rotations of the underlying $\R^2$. If we show that this choice of $(x, x^*)$ do not satisfy condition \eqref{eq-diamond-contr} then, by symmetry, the remaining choices would not satisfy \eqref{eq-diamond-contr} neither, and this would give us the desired contradiction.

Indeed, the pair $(z, z^*)\in \Pi(X)$ that do not satisfy \eqref{eq-diamond-contr} for $x = (1 - \eps, 1,  0)$, $x^*=(\eps/2, 1 - \eps/2, 0)$ is the following one:
$z = (1 - \eps, 1, \eps/2)$, $z^*= (\eps/2, 1 - \eps/2, \eps)$. Let us check the required properties. At first, $z = A_2 \in S_X$. Then, $z^*(z) = 1$. The last property means, that $\|z^*\| \geq 1$, so in order to check that $\|z^*\| = 1$ it remains to show that $|z^*(A_k)| \leq 1$ for all $k$. This is true for $\eps < 1$. Finally, $\norm{z-x} = \norm{(0, 0, \eps/2)} = \frac\eps2 \norm{\frac43 A_1} = \frac23 \eps < \eps$, and $\norm{z^*-x^*} = \norm{(0, 0, \eps)} = \langle(0, 0, \eps), A_1\rangle = \frac34 \eps < \eps$.
\end{proof}

\end{document}